\documentclass[12pt]{amsart}
\usepackage{amsmath,amsthm,amsfonts,amssymb}
\usepackage{graphicx,color}

\newtheorem{theorem}{Theorem}
\newtheorem{proposition}[theorem]{Proposition}

\newtheorem{lemma}[theorem]{Lemma}

\numberwithin{equation}{section}

\begin{document}

\title[Hitting probabilities of a flow]{Hitting probabilities of a Brownian flow with Radial Drift}
\author[Lee]{Jong Jun Lee}
\author[Mueller]{Carl Mueller}
\thanks{The second author was partially supported by a Simons Foundation grant.}
\author[Neuman]{Eyal Neuman}

\address{Jong Jun Lee: Dept. of Bioinformatics
\\University of Texas Southwestern Medical Center
\\Dallas, TX 75390}
\email{JongJun.Lee@UTSouthwestern.edu}

\address{Carl Mueller: Dept. of Mathematics
\\University of Rochester
\\Rochester, NY  14627 }
\urladdr{http://www.math.rochester.edu/people/faculty/cmlr}

\address{Eyal Neuman: Dept. of Mathematics
\\Imperial College London
\\ London, UK  SW7 2AZ}
\urladdr{http://eyaln13.wixsite.com/eyal-neuman}


\begin{abstract}
We consider a stochastic flow $\phi_t(x,\omega)$ in $\mathbb{R}^n$ with 
initial point $\phi_0(x,\omega)=x$, driven by a single $n$-dimensional 
Brownian motion, and with an outward radial drift of magnitude $\frac{
F(\|\phi_t(x)\|)}{\|\phi_t(x)\|}$, with $F$ nonnegative, bounded and Lipschitz.
We consider initial points $x$ lying in a set of positive distance from the 
origin.  We show that there exist constants $C^*,c^*>0$ not depending on $n$, 
such that if $F>C^*n$ then the image of the initial set under the flow has 
probability 0 of hitting the origin.  If $0\leq F \leq c^*n^{3/4}$, and if the 
initial set has nonempty interior, then the image of the set has positive 
probability of hitting the origin.  
\end{abstract}
\maketitle

\section{Introduction and Main Results}
\label{section:introduction}

In this paper we study a hitting problem of stochastic flows with radial drift. Let $(\Omega,\mathcal F,P)$ be a probability space.
For $x\in\mathbb{R}^n\setminus\{0\}$ we define the unit radial vector
in the $x$ direction as
\begin{equation*}
\mathbf{u}(x)=\frac{x}{\|x\|}
\end{equation*}
and we define
\begin{equation*}
\mathbf{u}(0)=0.
\end{equation*}
Here $\|\cdot\|$ denotes the Euclidean norm.  

We consider a Brownian flow 
$\phi=\{\phi_{t}(x,\omega); \, t\geq0, \,x\in \mathbb{R}^{n},\omega\in\Omega\}$
starting from a set  $\mathcal A_{0}\subset \mathbb{R}^n$ such that 
\begin{align}
\label{eq:flow}
d\phi_{t}(x) &=\frac{F(\|\phi_t(x)\|)}{\|\phi_{t}(x)\|}\mathbf{u}(\phi_t(x))dt +dW^{}_{t} \quad \textrm{for } t\geq 0, \\
\phi_0(x)&= x \in \mathcal A_{0}.  \nonumber
\end{align}
Here, $W_{t}$ is a standard $n$-dimensional Brownian motion, not depending 
on $x$, and $F(\cdot)$ is a nonnegative, bounded Lipschitz function.  
Throughout the paper the differential $d$ is taken with respect to $t$.  We 
also set up the standard Brownian filtration
\begin{equation*}
\mathcal{F}_t=\sigma(W_s: s\leq t).
\end{equation*}

Note that the flow might not be defined after $\phi$ reaches the origin, 
due to the singular drift there.  In fact, we will use a stopping time which 
is intuitively defined as  
\begin{equation} \label{tau-def-phi} 
\tau=\tau_{\mathcal{A}_0}
=\inf\{t: \phi_t(x)=0\text{ for some $x\in\mathcal{A}_0$}\}.
\end{equation}
For additional details we refer the reader to Section \ref{section:setup}.

We are interested in the question of whether the flow can hit the origin 
with positive probability, in the case where the initial set $\mathcal{A}_0$ 
is at positive distance from the origin.  We say that $\phi$ hits 
0 if $\tau<\infty$, and define the corresponding event
\begin{equation*}
G_{\mathcal{A}_0,F,n}:=\{\tau_{\mathcal{A}_0}<\infty\}.
\end{equation*}

Before stating our main results, we give some background on related hitting problems. 
It is well known that for a single Brownian motion $W_t$ in $\mathbb{R}^n$, 
the radial distance $D_t=\|W_t\|$ is a Bessel process which satisfies the 
following SDE, 
\begin{equation}
\label{eq:bessel}
dD_{t}=d\tilde{W}_{t}+\frac{n-1}{2D_{t}}dt,
\end{equation}
where $\tilde{W}_t$ is a one-dimensional Brownian motion.  Assuming 
$D_0>0$, it is a familiar fact that $D_t$ can hit 0 with positive 
probability iff $n<2$, where fractional values of $n$ are allowed in 
(\ref{eq:bessel}). This question and many more can be answered using ideas 
from potential theory and harmonic functions;  see \cite{mper10} or most other
books in Markov processes.  In one dimension and with 
$F(y)=(n-1)/2$, we see that \eqref{eq:flow} is identical to the Bessel 
equation \eqref{eq:bessel}.

There is an intimate connection between Bessel processes and their 
associated flows, to stochastic Loewner evolution (SLE), also called 
Schramm-Loewner evolution (see Lawler \cite{law05} for some basic facts).   
Indeed, on page x of the preface of \cite{law05}, Lawler states ``With the 
Brownian input, the Loewner equation becomes an equation of Bessel type, 
and much of the analysis of SLE comes from studying such stochastic 
differential equations. For example, the different ``phases'' of SLE 
(simple/non-simple/space-filling) are deduced from properties of the Bessel 
equation''.  In his St. Flour notes \cite{wer04}, Werner states on page 131:
``Then, we see that SLE$_\kappa$ can be interpreted in terms of the flow of a 
complex Bessel process''.  For a further explanation of this point of view, 
see Katori \cite{kat15}.  

Here are a few details.  SLE is usually thought of as a flow in a subset of 
the complex plane, with random parameters.  The equation for chordal SLE 
$g_t(x)$, which takes values in the upper complex half-plane, is
\begin{equation*}
\partial_tg_t(x)=\frac{2}{g_t(x)-\kappa W(t)},
\end{equation*}
where $W(t)$ is a one-dimensional Brownian motion and $\kappa>0$ is a 
parameter.  Setting $g_t(x)=\kappa h_t(x)+\kappa W(t)$, we find
\begin{equation*}
dh_t(x)=\frac{2/\kappa^2}{h_t(x)}dt-dW(t),
\end{equation*}
which is an equation of Bessel type.  

Since the Bessel process and its associated flow play such an important role 
in SLE, we feel that it is of interest to study Bessel flows in higher 
dimensions, and similar flows such as in \eqref{eq:flow}.  Perhaps the most 
basic property of the Bessel process is its probability of hitting the 
origin, and this is the question we investigate in this paper.  

Although such hitting questions are classical and have been completely 
answered for a large class of Markov processes, the reader may be surprised 
to learn that for processes such as $\phi_t(\cdot)$ which take values in a 
function space, the potential theory is more difficult or even intractable.  
In such cases we must fall back on more basic ideas, such as covering 
arguments and comparison with a random walk.  As usual, the critical case is 
the most difficult.  The critical case is where the parameters of the 
process are at or near the boundary between hitting and not hitting.  

Stochastic partial differential equations (SPDE) provide a source of such 
examples, and hitting questions for such equations have been studied in 
\cite{DKN07,DS10,DS15,mt01,NV09} among others.  These papers 
deal with the stochastic heat and wave equations either with no drift or with 
well behaved drift.  Since (\ref{eq:flow}) has singular drift,
the following result might be more relevant to our situation.  Suppose that 
$x$ lies in the unit circle $[0,1]$ with endpoints identified, and that 
$u(t,x)$ satisfies 
\begin{equation}
\label{eq:mueller-pardoux}
\partial_tu=\partial_x^2u+u^{-\alpha}+\dot{W},
\end{equation}
where $\dot{W}=\dot{W}(t,x)$ is two-parameter white noise.  Assume that 
$u(0,x)$ is continuous and strictly greater than 0.  Then $u$ hits 0 with 
positive probability if $\alpha<3$, and hits 0 with probability 0 if 
$\alpha>3$.  We say that $u$ hits 0 if there is a point $(t,x)$ such 
that $u(t,x)=0$.  See \cite{mue97} and \cite{mp99} for details.  
A natural question about (\ref{eq:mueller-pardoux}) is whether white noise 
$\dot{W}$ could be replaced by colored noise.  We can regard (\ref{eq:flow}) 
as a degenerate SPDE which does not have the Laplacian, and where the 
colored noise is independent of $x$, so the noise is at the opposite extreme 
from white noise.  

Returning to our stochastic flow $\phi_t$, we see that for a fixed point 
$x_0\in\mathbb{R}^n$ and for $n\geq2$ there is no chance that $\phi_t(x_0)$ 
can hit 0.  In that situation, even with zero drift $F\equiv0$, as already 
mentioned, $\|\phi_t(x_0)\|$ is a Bessel process with radial drift 
$(n-1)/(2\|\phi\|)$ and $n\geq2$.  Adding extra radial drift with $F\geq0$ 
makes it even less likely for $\phi_t(x_0)$ to hit the origin. However, if 
we allow $x\in\mathcal{A}_0$ to vary, then we may have several chances for 
$\phi_t(x)$ to hit 0. Intuitively, the Brownian motion $W_t$ which drives 
our flow has $n$ independent components, so we may expect that $\phi$ has $n$ 
independent chances to get closer to 0, and so the critical drift should be 
proportional to $n$.  Our first guess might also be that the critical drift 
is proportional to $1/\|\phi\|$, as for the Bessel process.  Here, the 
critical drift is that drift which lies on the boundary between a positive 
probability of hitting 0 and a zero probability of hitting 0.  
Unfortunately, there is no comparison principle for $\|\phi_t(x)\|$, so we 
cannot be sure that increasing the radial drift increases $\|\phi\|$. Thus 
we cannot conclude that larger radial drift leads to a smaller probability 
of $\phi$ hitting 0, and so we cannot give precise meaning to the term 
``critical drift''.

By introducing new processes $\psi_t(x)=\phi_t(x)-W_t$ and $B_t=-W_t$, we may 
rewrite the problem in terms of $\psi_t$ and $B_t$:
\begin{align}
\label{main_equation}
d\psi_{t}(x) &= \frac{F(\|\psi_t(x)-B_t\|)}{\|\psi_{t}(x)-B_{t}\|}\mathbf{u}(\psi_{t}(x)-B_{t})dt, \quad 
\textrm{for } 0\leq t<\tau, \\
\psi_0(x) &= x \in \mathcal A_{0}.  \notag
\end{align}
Note that $B_t$ is a standard $n$-dimensional Brownian motion.  
We denote 
\begin{equation*}
\mathcal A_{t} = \{\psi_{t}(x): \ x\in \mathcal A_{0} \},
\end{equation*}
which for any $t<\tau$ is a subset of $\mathbb{R}^{n}$. 

Now we are ready to state our main results. Recall that 
$F:\mathbb{R}^n\to\mathbb{R}$ is a nonnegative Lipschitz function.   
In our first theorem, we prove that for $F$ bounded below by a large
enough constant, the Brownian motion $B_t$ does not hit the 
region $\mathcal{A}_t$. Moreover, the distance between $\mathcal{A}_t$ and $B_t$ tends to $\infty$ as $t\to\infty$.

\begin{theorem}  
\label{thm_not_hitting}  
There exists $C^{*}>0$ not depending on $n$ such that for all $n\geq1$
the following holds.  Suppose the set $\mathcal A_{0} \subset \mathbb{R}^n$ 
has a positive distance from the origin.  If $F(x)\geq C^{*}n$ for all 
$x\in\mathbb{R}^n$,  then we have 
\begin{equation*}
P(G_{\mathcal{A}_0,F,n}) =0. 
\end{equation*}
\end{theorem} 

In the next two theorems we prove that when $F$ is bounded from above by a small 
enough constant, the Brownian motion $B_t$ hits the region $\mathcal{A}_t$ 
with positive probability.  Hitting the region with probability one requires 
that the region be rather large, or else the transience of Brownian motion 
in high dimensions works against hitting.  
\begin{theorem}  
\label{theorem-hit}
There exists $c^{*}>0$ not depending on $n$ such that for all $n\geq 1$ the 
following holds.  Suppose 
\begin{equation*}
\mathcal A_{0}= \{(x_{1},...,x_{n}):x_{1}\geq1\text{ and }x_{2},...x_{n}\in\mathbb R\}. 
\end{equation*}
If $0\leq F(x) \leq c^{*}n^{3/4}$ for all  $x\in\mathbb{R}^n$, then
\begin{equation*}
P(G_{\mathcal{A}_0,F,n})=1. 
\end{equation*}
\end{theorem} 
If we are willing to accept a positive probability of $\{\phi_t(x): x\in\mathcal{A}_0\}$ hitting 0, we can extend Theorem \ref{theorem-hit} to a broad class of initial sets $\mathcal{A}_0$.  If we make the set $\mathcal{A}_0$ smaller, then the set of $\omega$ for which $\{\phi_t(x): x\in\mathcal{A}_0\}$ hits 0 will also become smaller.   Theorem \ref{theorem-hit2} shows that if $\mathcal{A}_0$ contains an open ball, then the probability of the above set hitting 0 is still positive.\begin{theorem}  
\label{theorem-hit2}
\label{theorem-hit-with-general-initial-set}
There exists $c^{*}>0$ not depending on $n$ such that for all $n\geq 1$ the 
following holds.  Suppose 
$\mathcal A_{0}\subset\mathbb{R}^n$ has nonempty interior.  
If $0\leq F(x) \leq c^{*}n^{3/4}$ for all  $x\in\mathbb{R}^n$, then
\begin{equation*}
P(G_{\mathcal{A}_0,F,n})>0. 
\end{equation*}
\end{theorem} 

Roughly speaking, our theorems say that the critical drift (if there is one) 
for (\ref{eq:flow}) has magnitude between $c^*n^{\frac{3}{4}}/\|\phi\|$ and
$C^*n/\|\phi\|$.  

We also note that in the case $n=1$, our question reduces to finding the 
critical parameter for a Bessel process to hit the origin.  We leave the
details of this case to the reader.

\section{Precise setup of the flow}
\label{section:setup}

A good reference for the general theory of stochastic flows is the book of 
Kunita \cite{kun90}.  However, since our flow is driven by a single Brownian 
motion, we can make use of properties of Brownian motion to simplify our 
setup.  We also need to keep in mind that the flow should only exist up to 
time $\tau$ which was intuitively defined as the first time $t$ that 
$\phi_t(x)=0$ for some $x\in\mathcal{A}_0$.  

Our strategy is to truncate the drift and take the limit as the truncation 
level tends to infinity.  Let
\begin{equation}
\label{eq:def-f-N}
f_N(x)=\frac{F(\|x\|)}{\|x\|\vee (1/N)}\mathbf{u}(x),
\qquad
f_\infty(x)=\frac{F(\|x\|)}{\|x\|}\mathbf{u}(x),
\end{equation}
where $\vee$ denotes the maximum.  Observe that $f_N$ is a globally 
Lipschitz function.  

We start by defining the translated flow \eqref{main_equation}, 
with respect to the truncated drift $f_N$.  Let $\psi^{(N)}_t(x)$ satisfy
\begin{align}
\label{main_equation_truncated}
d\psi^{(N)}_{t}(x) &= f_N\big(\psi^{(N)}_{t}(x)-B_{t}\big)dt, \quad 
\textrm{for } t\geq0, \\
\psi^{(N)}_0(x) &= x \in\mathbb{R}^n  \notag
\end{align}
with $B_t=-W_t$.  Then \eqref{main_equation_truncated} is an ordinary 
differential equation and $f_N\big(\psi^{(N)}_{t}(x)-B_{t}\big)$ is 
a (random) Lipschitz function of $\psi^{(N)}_t(x)$.  It follows that 
\eqref{main_equation_truncated} has a unique solution which is 
$\mathcal{F}_t$-adapted, where we recall that $\mathcal{F}_t$ is the 
Brownian filtration.  

Given a set $\mathcal{A}\subset\mathbb{R}^n$, let 
\begin{equation*}
\tau^{(N)}(\mathcal{A})
=\inf\Big\{t\geq0: \inf_{x\in\mathcal{A}}|\psi^{(N)}_t(x)-B_t|\leq1/N\Big\}
\end{equation*}
and let $\tau^{(N)}(\mathcal{A})=\infty$ if the above set is empty.  
It follows that for $x\in\mathcal{A}$ and 
$0\leq t\leq\tau^{(N)}(\mathcal{A})$, we have
\begin{align}
\label{eq:psi-1}
d\psi^{(N)}_{t}(x) &= f_N\big(\psi^{(N)}_{t}(x)-B_t\big)dt  \\
\psi^{(N)}_0(x) &=  x \in\mathbb{R}^n  \notag
\end{align}
and that $f_N(\psi^{(N)}_t(x))=f_\infty(\psi^{(N)}_t(x))$ for 
$0\leq t\leq\tau^{(N)}({\mathcal{A}})$, almost surely.  Again almost 
surely, we have
$0\leq\tau^{(1)}(\mathcal{A})\leq\tau^{(2)}(\mathcal{A})\leq\cdots$.  
Now we give the rigorous definition of 
$\tau(\mathcal{A})$ , namely
\begin{equation*}
\tau(\mathcal{A})=\lim_{N\to\infty}\tau^{(N)}(\mathcal{A}).
\end{equation*}
Let $\tau(\mathcal{A})=\infty$ on the exceptional set where 
$\tau^{(N)}(\mathcal{A})$ is not nondecreasing in $N$.  We note that 
$\tau(\mathcal{A})=\infty$ is possible even off of this exceptional set.  Now 
we can define $\psi^{\mathcal{A}}$ as follows for $0\leq t<\tau(\mathcal{A})$ 
and $x\in\mathcal{A}$
\begin{equation*}
\psi^{\mathcal{A}}_t(x) = \lim_{N\to\infty}\psi^{(N)}_t(x).
\end{equation*}
We see that for $0\leq t<\tau(\mathcal{A})$ and for $x\in\mathcal{A}$, we have
\begin{align}
\label{eq:psi-2}
d\psi^\mathcal{A}_{t}(x) &= f\big(\psi^{\mathcal{A}}_{t}(x)-B_t\big)dt  \\
\psi^{\mathcal{A}}_0(x) &= x \in \mathbb{R}^n. \notag
\end{align}
In order to define $\psi^\mathcal{A}_t$ for all time, we introduce a 
cemetery state $\Delta$ and let
\begin{equation*}
\psi^\mathcal{A}_t=\Delta
\end{equation*} 
for $t\geq\tau(\mathcal{A})$.  Finally, as in the introduction, we define
\begin{equation*}
\mathcal A_{t} = \{\psi_{t}(x): \ x\in \mathcal A \}.
\end{equation*}
  
Our next goal is to formulate a strong Markov property for 
$\psi^{\mathcal{A}}$.  In fact $\psi^{\mathcal{A}}$ 
is a nonanticipating measurable function of the Brownian path $B_t$, so we can 
define the shift operator $\theta_t$ for $B_t$ and extend it to 
$\psi^{\mathcal{A}}$.   

\begin{lemma}
\label{lem:strong-Markov}
Let $a\subset\mathbb{R}^n$ and suppose that $\sigma$ is a stopping time 
with respect to the filtration $(\mathcal{F}_t)_{t>0}$ generated by Brownian 
motion $B_t$.  Assume that $\sigma<\tau(\mathcal{A})$ almost surely.  Then 
conditioned on $\mathcal{F}_\sigma$, we have that 
$\psi_t(x):=(\theta_\sigma\psi^{\mathcal{A}})_t(x)$ satisfies the following 
modified version of \eqref{eq:psi-2}.  For 
$0\leq t<\theta_\sigma\tau(\mathcal{A})=\tau(\mathcal{A}_\sigma)$ and 
$x\in\mathcal{A}_\sigma$,
\begin{align}
\label{eq:psi-3}
d\psi_{t}(x) &= f\big(\psi_{t}(x)-(\theta_\sigma B)_t\big)dt  \\
\psi_0(x) &= x.     \notag
\end{align}
\end{lemma}

This lemma follows immediately from the strong Markov property of 
Brownian motion.  

For $x\in\mathcal{A}$ and $t<\tau(\mathcal{A})$, we can also define
\begin{equation*}
\phi^\mathcal{A}_t(x)=\psi^{\mathcal{A}}_t(x)-B_t
\end{equation*}
and let $\phi^\mathcal{A}_t(x)=\Delta$ for $t\geq\tau_\mathcal{A}$.  From 
\eqref{eq:psi-3} we see that $\phi^\mathcal{A}_t(x)$ satisfies 
\eqref{eq:flow} for $x\in\mathcal{A}$ and $t<\tau_\mathcal{A}$.  
Also, the strong Markov property for $\phi^\mathcal{A}_t(x)$ follows from 
Lemma \ref{lem:strong-Markov}, and the formulation is similar.  

We will often use the following containment property, which follows by the 
definition of a flow.  If $S_1\subset S_2\subset\mathcal{A}_0$ and 
$t\in[0,\tau_{\mathcal{A}_0})$, then we have that with probability 1,
\begin{align}
\label{eq:subset-property}
\phi_t(S_1)\subset\phi_t(S_2)
\qquad \text{and} \qquad \psi_t(S_1)\subset\psi_t(S_2).
\end{align}

Finally, we record a scaling property of (\ref{main_equation}) which we will 
use repeatedly.  Since this kind of scaling is standard, we leave the proof 
to the reader.  
\begin{lemma}
\label{lem:scaling}
Let $(\psi_t,B_t,\mathcal{A}_0,\tau)$ be a solution to (\ref{main_equation}).  Then for $\lambda > 0$, 
\begin{equation*}
(\lambda \psi_{\lambda^{-2}t},\lambda B_{\lambda^{-2}t},\lambda\mathcal{A}_0,\lambda^{-2}\tau)
\end{equation*}
is also a solution to (\ref{main_equation}), but with $F(x)$ replaced by 
$F(x/\lambda)$.
\end{lemma}

Note that this scaling does not change any lower or upper bounds on $F$.

\section{Proof of Theorem~\ref{thm_not_hitting}}
\label{sec:not-hit}

Throughout this section we assume $n\geq2$;  as mentioned at the end of 
Section \ref{section:introduction}, the case $n=1$ reduces to an easy 
question about the Bessel process.  

Here is an outline of our strategy.  We will show that if there exists a 
constant $C^*$ such that if $F(x)>C^*n$ for all $x\in\mathbb{R}^n$ then the 
following holds.  We construct a sequence of stopping times 
$\tau_0\leq\tau_1\leq\cdots$ such that $\lim_{i\to\infty}\tau_i=\infty$ 
almost surely, and such that the distance from $B_{\tau_i}$ to 
$\mathcal{A}_{\tau_i}$ tends to get larger and larger.  Furthermore, at 
times $\tau_i$ we enlarge the region $\mathcal{A}_{\tau_i}$, and call the 
enlarged region $\mathbb{A}_{\tau_i}$.  Then we show that with probability 
1, the Brownian motion $B_t$ never hits the enlarged region $\mathbb{A}_t$.  
So by (\ref{eq:subset-property}) and the strong Markov property, $B_t$ 
cannot hit the original region.  We write $\rho_{i}$ for the distance from 
$B_{\tau_i}$ to our enlarged region $\mathbb{A}_{\tau_i}$.  This would prove 
Theorem \ref{thm_not_hitting}.

Now we give the details.  To start with, let $\tau_0=0$.  Considering that 
the initial domain $\mathcal{A}_0$ of the flow is of positive distance 
$\rho_0>0$ from the origin, we define $\mathbb{A}_0=\mathbf{B}_{\rho_0}(0)^c$, 
that is, the complement of the ball of radius $\rho_0$ centered at the 
origin.  Then $\mathcal{A}_0\subset\mathbb{A}_0$.  We provisionally define 
$\tilde{\mathbb{A}}_t=\psi_t^{\mathbb{A}_0}(\mathbb{A}_0)$ for $t<\tau(\mathbb{A}_0)$, 
where we recall that $\tau(\mathbb{A}_0)$ is the stopping time $\tau$ 
defined with respect to the initial set $\mathbb{A}_0$.  Next, let 
$\tilde{\rho}(t)$ be the shortest distance from $B_t$ to 
$\tilde{\mathbb{A}}_t$.  Let $\tau_1$ be the smallest time $t\in[0,\tau(\mathbb{A}_0))$ 
such that $\tilde{\rho}(t)$ equals either $\rho_0/2$ or $2\rho_0$.  If there 
is no such time $t$, then $B_t$ never hits $\tilde{\mathbb{A}}_t$, and 
Theorem \ref{thm_not_hitting} is proved.  So we assume that there is such a 
time $\tau_1$.  Finally, let $\rho_1=\tilde{\rho}(\tau_1)$.  

For $0=\tau_0\leq t<\tau_1$, define 
$\mathbb{A}_t=\tilde{\mathbb{A}}_t=\psi_t^{\mathbb{A}_0}(\mathbb{A}_{\tau_0})$.  Next, define
\begin{equation*}
\mathbb{A}_{\tau_1}=\mathbf{B}_{\rho_{1}}(B_{\tau_1})^c.  
\end{equation*}
So at the end of our time stage $[\tau_0,\tau_1)$ we have changed our 
region so it is again the complement of a ball centered at the location of 
the Brownian particle $B_{\tau_1}$ and of radius $\rho_{1}$. 
We illustrate this setup in Figure~\ref{fig_setup}.

\begin{figure}[!h]
	\centering
	\setlength{\unitlength}{0.1\textwidth}
	\begin{picture}(10,6)
	\put(0,0){\includegraphics[width=\textwidth]{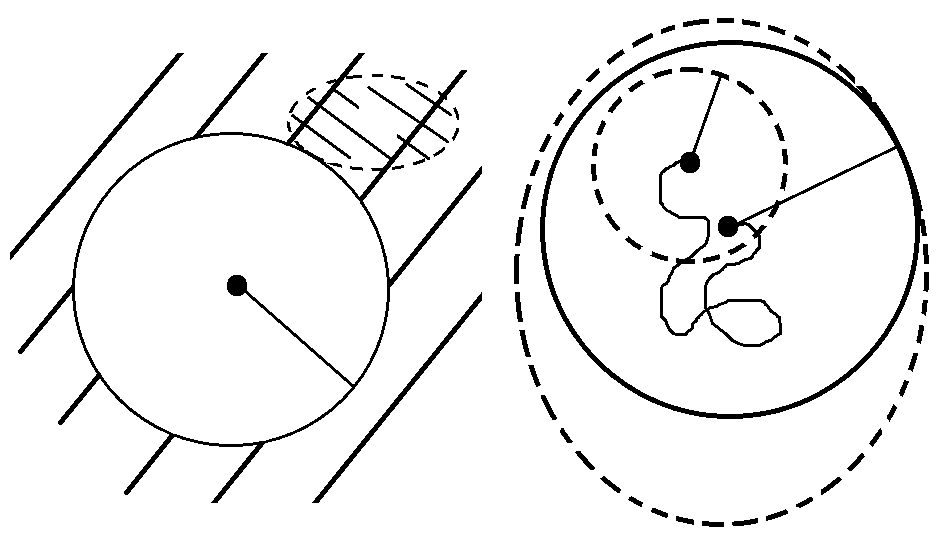}}
	\put(3.8,4.4){$\mathcal{A}_0$}
	\put(1.7,4.6){$\mathbb{A}_0$}
	\put(2.8, 1.8){$\rho_0$}
	\put(7.6, 4.4){$\rho_0$}
	\put(8.5, 3.5){$\rho_1 = 2 \rho_0$}
	\put(6.9, 4.2){$B_0$}
	\put(7.6, 3.55){$B_{\tau_1}$}
	\put(6.4, 1){$\partial\mathbb{A}_{\tau_1}=\partial \mathbf{B}_{\rho_1}(B_{\tau_1})$}
	\put(9.4, 0.8){$\partial\tilde{\mathbb{A}}_{\tau_1}$}
	\end{picture}
	\caption{The initial set $\mathbb{A}_0$ and the case when $\rho_1 = 2\rho_0$.}
	\label{fig_setup}
\end{figure}
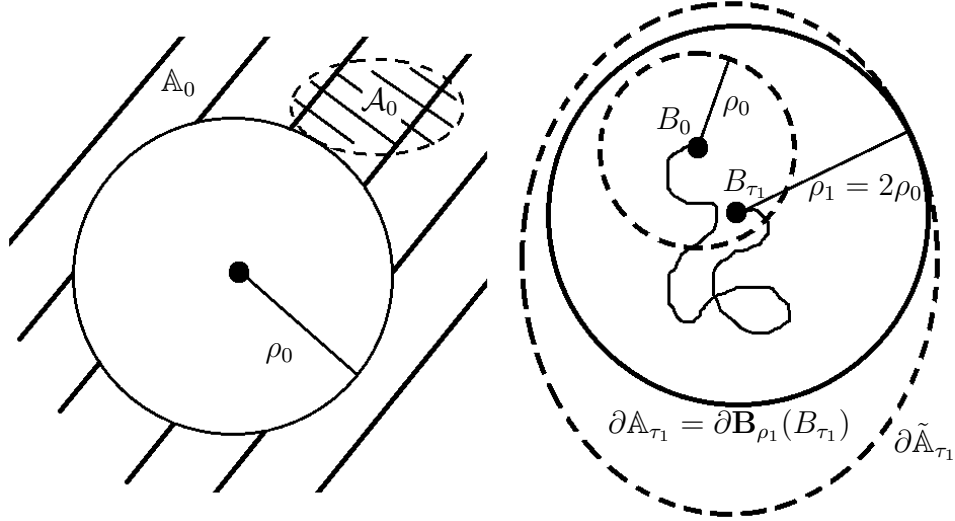

Now we repeat the procedure, with time restarted at $\tau_1$ and with 
$B_{\tau_1}$ playing the role of the origin, and proceed inductively.  Here 
we have used the strong Markov property to restart the process.  To 
summarize, at each time $\tau_i$ we change our region to be 
$\mathbb{A}_{\tau_i}=\mathbf{B}_{\rho_{i}}(B_{\tau_i})^c$, the 
complement of the ball of radius $\rho_{i}$ with center $B_{\tau_i}$. 
It follows from our construction that $\rho_{{i}}$ equals either 
$\rho_{i-1}/2$ or $2\rho_{i-1}$. Observe that the following lemma implies Theorem \ref{thm_not_hitting}.
\begin{lemma}
\label{lem:properties-rho-tau}
There exists a constant $C^*>0$ not depending on $n$ such that if 
$F(x)\geq C^*n$ for all $x\in\mathbb{R}^n$, then
\begin{enumerate}
\item[(i)] $\lim_{i\to\infty}\rho_i=\infty$  almost surely. 
\item[(ii)] $\lim_{i\to\infty}\tau_i=\infty$ almost surely.
\end{enumerate}
\end{lemma}
The remainder of this section is therefore dedicated to the proof of Lemma \ref{lem:properties-rho-tau}. 
Before we start with the proof, we introduce the following auxiliary lemmas, which will be proved later in this section.

The following lemma establishes a lower bound on the tail distribution of the Brownian exit time from a ball of radius $1/2$.  
\begin{lemma}
\label{lemma_1} 
Let $\sigma$ be the first time that an $n$-dimensional standard Brownian motion $B_t$ reaches $\|B_t\|=1/2$,  and assume 
$0\le p<1$. 
There exists a constant $C_1>0$ not depending on the dimension $n\ge 1$ such that 
\begin{equation}
\label{eq:first-goal}
P\left(\sigma>\frac{C_1}{n}\right)>p.
\end{equation}
\end{lemma}
The above probability is computed explicitly in Ciesielski and Taylor 
\cite{ciesielski-taylor62}, Theorem 2, page 444.  However, their expression 
involves an infinite series with roots of Bessel functions, and dependence 
on the dimension $n$ is not immediately clear.  So we give a self-contained 
proof.

Here we make a comment about the difficulties in working in 
high-dimensional space.
We want to establish a lower bound for the probability that $B$ stays 
inside an $n$-dimensional ball of radius $1/2$.  Therefore it is not 
enough to bound 
\begin{equation*}
P\left(\sigma>\frac{C_1}{n}\right) = P\left(\sup_{0 \le t \le C_1/n}\|B_t\|<\frac{1}{2}\right) 
\end{equation*}
by the probability that $B$ remains inside the cube $[-1/2,1/2]^n$, 
due to the difference between the volume of balls and cubes in high
dimensions.   Also, a high-dimensional cube has a large ratio between the 
maximum and minimum distances between the center and the boundary.  
In the proof of Lemma \ref{lemma_1}, we use the radial component 
$D_t = \|B_t\|$ of the Brownian motion and introduce stopping times and 
auxiliary processes for comparison, in order to derive (\ref{eq:first-goal}).
\\

In the following lemma we establish a lower bound on the transition probability for 
$\rho_i$. 
\begin{lemma}
\label{lem:lower-bound-two-thirds}
There exists a constant $C^*$ not depending on $n$ such that if 
$F(x)\geq C^*n$ for all $x\in\mathbb{R}^n$, then for all $i\in\mathbb{N}$,
\begin{equation*}
P\left(\rho_{i}=2\rho_{i-1}\big|\mathcal{F}_{\tau_{i-1}}\right)\geq\frac{2}{3}.
\end{equation*}
\end{lemma}
Now we are ready to prove Lemma \ref{lem:properties-rho-tau}.  
\begin{proof}[Proof of Lemma \ref{lem:properties-rho-tau}]

(i) We define
\begin{equation*}
X_i=\log_2(\rho_i/\rho_{i-1})
\end{equation*}
where $X_i$ takes the values $\{-1,1\}$. Lemma 
\ref{lem:lower-bound-two-thirds} states that 
\begin{equation*}
P\left(X_i=1\big|\mathcal{F}_{\tau_{i-1}}\right)\geq\frac{2}{3}.
\end{equation*}
By expanding the probability space if necessary, we can construct a 
$\mathcal{F}_{\tau_i}$-measurable random variable $Y_i\leq X_i$ for all $i$, 
such that $Y_i$ also takes on the values $1,-1$, and such that 
\begin{equation*}
P\left(Y_i=1\big|\mathcal{F}_{\tau_{i-1}}\right)=\frac{2}{3}.
\end{equation*}
It follows that $Y_1,Y_2,\ldots$ is a sequence of i.i.d. random variables 
with expectation $1/3$.  The strong law of large numbers along with the 
fact that $X_i\geq Y_i$ implies that almost surely,
\begin{equation*}
\liminf_{m\to\infty}\log_2\left(\frac{\rho_m}{\rho_0}\right)
=\liminf_{m\to\infty}\sum_{i-1}^{m}X_i
\geq\lim_{m\to\infty}\sum_{i-1}^{m}Y_i=\infty
\end{equation*}  
and we get part (i) 

(ii) Since $(\tau_i)_{i\in\mathbb{N}}$ is a nondecreasing sequence of random variables, the limit $\tau_\infty=\lim_{i\to\infty}\tau_i$ almost surely exists, and we would like to show that it is infinity a.s.
By the construction of the sequence $\{\rho_i\}_{i\geq 0}$ and since $\rho_i\to\infty$ a.s. by part (i) the distance between $B_t$ and $\mathcal{A}_t$ tends to $\infty$ as $t\uparrow\tau_\infty$ and is strictly positive for $t<\tau_\infty$ with probability one.
  
Therefore, for almost every realization $\omega$, we can find a natural number $N(\omega)$ such that $\frac{1}{N} \le \psi^{\mathbb{A}_0}_t(x)$ for all $0\le t < \tau_\infty$ and $x \in\mathbb{A}_0$. On this $\omega$, the drift in equation (\ref{main_equation})  is therefore identical to $f_N$ in (\ref{eq:def-f-N}).
Since $f_N$ is a Lipschitz function, the corresponding flow is finite for all time with probability one.  Therefore, the distance between $B_t$ and $\mathcal{A}_t$ tends to $\infty$ as 
$t\uparrow\tau_\infty$, only if $\tau_\infty=\infty$ a.s. It follows that $\tau_i\to\infty$ almost surely, and this proves part (ii).
\end{proof}

By our earlier comment, Theorem~\ref{thm_not_hitting} follows from Lemma \ref{lem:properties-rho-tau}. The remainder of this section is dedicated to the proofs of Lemmas \ref{lemma_1} and \ref{lem:lower-bound-two-thirds}.  

\begin{proof}[Proof of Lemma \ref{lemma_1}]
We first show that  there exist a natural 
number $n^*$ and a constant $C_{1,n^*}>0$ such that (\ref{eq:first-goal}) holds 
whenever $n \ge n^*$.  Then, we generalize the requirement $n \ge n^*$ to $n\ge 1$ by possibly choosing a number $C_1>0$ smaller than $C_{1,n^*}$.

Recall that $D_t=\|B_t\|$ is a Bessel process which satisfies
\begin{equation}
\label{eq:Bessel}
dD_t=\frac{n-1}{2D_t}dt+d\tilde{W}_t,
\end{equation}
for some one-dimensional Brownian motion $\tilde{W}_t$.  

Let $\sigma_0$ be the first time $t$ that $D_t=3/8$, and let 
\begin{equation*}
D^{(1)}_t=D_{\sigma_0+t}.
\end{equation*}
Let $\mathcal{G}_t$ be the filtration generated by $\tilde{W}_t$.  
By the strong Markov property of Brownian motion, conditioned on 
$\mathcal{G}_{\sigma_0}$, $D^{(1)}_t$ satisfies 
(\ref{eq:Bessel}) with the initial condition $D^{(1)}_0=3/8$.  Now define 
$D^{(2)}_t$ to be the solution of
\begin{align*}
dD^{(2)}_t&=f(D^{(2)}_t)dt+d\tilde{W}_t,  \\
D^{(2)}_0&=\frac{3}{8}, 
\end{align*}
where 
\begin{equation*}
f(x)=
\begin{cases}
(n-1)/(2x), & \text{if $0<x<1/4$},  \\
2(n-1), & \text{if $x\geq1/4$}.
\end{cases}
\end{equation*}
By a standard comparison result (see Theorem 1.1 in Chapter V.1 of Ikeda and Watanabe \cite{iw89}), 
since $D^{(1)}_t$ has drift which is not greater than the drift of 
$D^{(2)}_t$, we conclude
\begin{equation*}
D^{(1)}_t\leq D^{(2)}_t
\end{equation*}
for all $t\geq0$ with probability 1.  Actually, the result of Ikeda and 
Watanabe does not cover locally unbounded drift, so we must argue via a 
truncation argument.  We leave these details to the reader.  

So $D^{(1)}_t$ will reach $1/2$ later than $D^{(2)}_t$ reaches $1/2$.  
Now let $\sigma_1$ be the first time $t$ that $D^{(1)}_t=1/2$, and 
let $\sigma_2$ be the first time $t$ that $D^{(2)}_t=1/2$.  
It follows that $\sigma\geq\sigma_1\geq\sigma_2$, and so
\begin{equation*}
P(\sigma>x)\geq P(\sigma_2>x).
\end{equation*}
Finally, let $\sigma_3$ be the first time $t$ that $D^{(2)}_t=1/4$ or
$D^{(2)}_t=1/2$.  Then $\sigma_2\geq\sigma_3$, and so
\begin{equation*}
P(\sigma>x)\geq P(\sigma_2>x)\geq P(\sigma_3>x).
\end{equation*}

Therefore, in order to prove (\ref{eq:first-goal}), we look for a lower bound on $P(\sigma_3>x)$.  To this end, let
$D^{(3)}_t$ be the solution of 
\begin{align*}
dD^{(3)}_t&=2(n-1)dt+d\tilde{W}_t, \\
D^{(3)}_0&=\frac{3}{8},
\end{align*}
and let $\sigma_4$ to be the first time $t$ that $D^{(3)}_t=1/4$ or 
$1/2$.  Then $\sigma_3=\sigma_4$.  Finally, let $\sigma_5$ be the 
first time $t$ that $D^{(3)}_t=1/2$.  Then we have
\begin{align*}
P(\sigma>x)&\geq P(\sigma_3>x)=P(\sigma_4>x)  \\
&\geq P\left(\sigma_4>x, D^{(3)}_{\sigma_4}=1/2\right). 
\end{align*}
Note that on $D^{(3)}_{\sigma_4}=1/2$, we have $\sigma_4=\sigma_5$, and also
$\{D^{(3)}_{\sigma_4}=1/2\}^c=\{D^{(3)}_{\sigma_4}=1/4\}$.  
Thus, we can continue the above inequality as follows.  
\begin{align}
\label{eq:lower-bound-on-sigma-prob}
P(\sigma>x)&\geq P\left(\sigma_4>x, D^{(3)}_{\sigma_4}=1/2\right) \\
&= P\left(\sigma_5>x, D^{(3)}_{\sigma_4}=1/2\right) \nonumber\\
&= P\left(\sigma_5>x\right)-P\left(\sigma_5>x, D^{(3)}_{\sigma_4}=1/4\right) 
 \nonumber\\
&\geq P\left(\sigma_5>x\right)-P\left(D^{(3)}_{\sigma_4}=1/4\right) .\nonumber
\end{align}

We first give an upper bound on $P(D^{(3)}_{\sigma_4}=1/4)$.
Recall that 
\begin{equation*}
h(x)=\frac{e^{-2(n-1)x}-e^{-(n-1)}}{e^{-(n-1)/2}-e^{-(n-1)}}
\end{equation*}
is a harmonic function for $D^{(3)}_t$, since the process $D^{(3)}_t$ has 
generator
\begin{equation*}
Gf(x)=2(n-1)f'(x)+\frac{1}{2}f''(x).
\end{equation*}
Furthermore, $h(1/4)=1$ and $h(1/2)=0$. We deduce that
\begin{align}
\label{sigma3bound2}
P\left(D^{(3)}_{\sigma_4}=1/4\right)=h(3/8)
&=\frac{e^{-3(n-1)/2}-e^{-2(n-1)}}{e^{-(n-1)}-e^{-2(n-1)}}  \\
&= e^{-(n-1)/2}\big(1+o(1)\big)    \nonumber
\end{align}
as $n \rightarrow \infty$.

We now focus on $P(\sigma_5>x)$. Assume that 
\begin{equation*}
x\leq 1/(8(n-1)).
\end{equation*}
Note that for $t\leq x$ we have $(n-1)t\leq 1/8$.  
Thus, for such values of $x$, using the reflection principle we have
\begin{align*}
P(\sigma_5>x)&\geq P\left(\sup_{0\leq t\leq x}\tilde{W}(t)<1/8\right)  \\
&= 1-P\left(\sup_{0\leq t\leq x}\tilde{W}(t) \geq 1/8\right)  \\
&= 1-2P(\tilde{W}(x) \geq 1/8).
\end{align*}
Using a standard Gaussian random variable $Z$, we get
\begin{align*}
P(\tilde{W}(x) \geq 1/8)&=P(\sqrt{x}Z \geq 1/8)=P\left(Z \geq \frac{1}{8\sqrt{x}}\right)\\
&\le \frac{4\sqrt{2}\sqrt{x}}{\sqrt{\pi}}\exp\left(-\frac{1}{128x}\right).
\end{align*}
See Durrett \cite{dur10}, Theorem 1.2.3, page 12, for the Gaussian estimate.  

With $x\leq1/(8(n-1))$, we have
\begin{equation*}
P(\tilde{W}(x) \geq 1/8)\le \frac{2}{\sqrt{\pi}}\frac{1}{\sqrt{n-1}}\exp\left(-\frac{n-1}{16}\right).
\end{equation*}
So, 
\begin{equation}
\label{sigma4bound}
P(\sigma_5>x)\ge 1-\frac{4}{\sqrt{\pi}}\frac{1}{\sqrt{n-1}}\exp\left(-\frac{n-1}{16}\right).
\end{equation}

Combining (\ref{eq:lower-bound-on-sigma-prob}), (\ref{sigma3bound2}), and 
(\ref{sigma4bound}), we can conclude that for large values of $n$ and 
$x\leq 1/(8(n-1))$, we get
$P(\sigma>x)> p$.  In particular,
\begin{equation*}
P\left(\sigma>\frac{1}{8n}\right)>p
\end{equation*}
for all $n \ge n^*$ for some $n^*$.

It remains to lower the requirement $n \ge n^*$ to $n \ge 1$. Note that 
for each fixed value of $n$, we have $P(\sigma>C/n)\to 1$ 
as the constant $C \downarrow 0$.  Thus for each $1 \le n < n^*$, we can find $C_{1,n}>0$ such that 
\begin{equation*}
P\left(\sigma>\frac{C_{1,n}}{n}\right)>p.
\end{equation*}
By choosing
\begin{equation*}
C_1=\min\left\{C_{1,1},\ldots,C_{1,n^*-1}, \frac{1}{8}\right\},
\end{equation*}
we can conclude that
\begin{equation*}
P\left(\sigma>\frac{C_1}{n}\right)>p
\end{equation*}
for all $n \ge 1$.
\end{proof}

\begin{proof}[Proof of Lemma \ref{lem:lower-bound-two-thirds}]
Notice that by the strong Markov property applied at time $\tau_{i-1}$, we can 
start afresh at that time and relabel $\tau_{i-1}$ as $\tau_0=0$.  By 
translating if necessary, we may assume that 
\begin{enumerate}
\item $i=1$
\item $\tau_0=0$
\item $B_{0}=0$
\end{enumerate}
Furthermore, by scaling time and space via Lemma \ref{lem:scaling}, we can 
assume that 
\begin{enumerate}
\item[(4)] $\rho_{0}=1$.  
\end{enumerate}
These transformations may change our drift $F$, but not the 
lower bound on $F$.  Thus, to prove Lemma \ref{lem:lower-bound-two-thirds}, 
it suffices to show
\begin{equation} 
\label{eq:suffices-two-thirds}
P(\rho_{1}=2)\geq\frac{2}{3}.
\end{equation}

Our intuitive idea is as follows.  We show that with high probability the 
time spent by the Brownian motion in $\mathbf{B}_{1/2}(0)$ is greater than 
$C_1/n$, for some constant $C_1$ not depending on $n$.  If this event 
occurs, then the magnitude of the total drift on points $y \in \mathcal A_{t} \cap \left(\mathbf{B}_{5/2}(0)\setminus\mathbf{B}_{1}(0)\right)$ is 
\begin{equation*}
\frac{F(\|y-B_t\|)}{\|y-B_{t}\|} \ge \frac{C^*n}{3}
\end{equation*}
within 
the time interval $[0,C_1/n]$.  The inequality was from the hypothesis $F(x)\geq C^*n$ of Lemma~\ref{lem:lower-bound-two-thirds}  and from the distance 
between a point in $\mathbf{B}_{5/2}(0)$ and the  Brownian motion in 
$\mathbf{B}_{1/2}(0)$ is at most $5/2+1/2=3$.  
Furthermore, we will see that 
the projection of this drift onto the outward radial direction at worst
reduces this drift by a multiplicative factor $c$. See Figure~\ref{fig_angle} for the illustration. Thus, it follows from 
(\ref{main_equation}) the total outward radial drift on a point $y \in \mathbf{B}_{5/2}(0)\setminus\mathbf{B}_{1}(0)$ for the time interval $[0,C_1/n]$ is at least
\begin{equation} 
\label{drift-he} 
\inf_{y \in 
\mathbf{B}_{5/2}(0)\setminus\mathbf{B}_{1}(0)}\left\{ \frac{F(\|y-B_t\|)}{\|y-B_{t}\|}\right\} \cdot \frac{C_1}{n}  \ge \frac{C^*n}{3}\cdot c\cdot\frac{C_1}{n}>3/2,
\end{equation}
if $C^{*}$ is large enough. Also observe that $3/2$ is the radial distance 
from $\mathbf{B}_{1}(0)$ to the boundary of $\mathbf{B}_{5/2}(0)$. Roughly 
speaking, this means that if $C^{*}$ is large enough, then any point in the 
region $\mathbf{B}_{5/2}(0)\setminus\mathbf{B}_{1}(0)$ will now be at least 
at distance $5/2$ from the origin, while the Brownian motion remains in 
$\mathbf{B}_{1/2}(0)$.  Note that the drift might be smaller than $3/2$
if the point exits the ball $\mathbf{B}_{5/2}(0)$ during 
the time interval $[0,C_1/n]$, but this is what we want to show anyway.  
Also, once it exits this ball, it never reenters.  So 
in both cases the point in $\mathbf{B}_{5/2}(0)\setminus\mathbf{B}_{1}(0)$ 
exits the ball $\mathbf{B}_{5/2}(0)$, and the distance between the Brownian 
particle and $\mathcal A_{t}$ increases by a factor of $2$.  

\begin{figure}[!h]
	\centering
	\setlength{\unitlength}{0.1\textwidth}
	\begin{picture}(10,7)
	\put(1,0){\includegraphics[width=90mm]{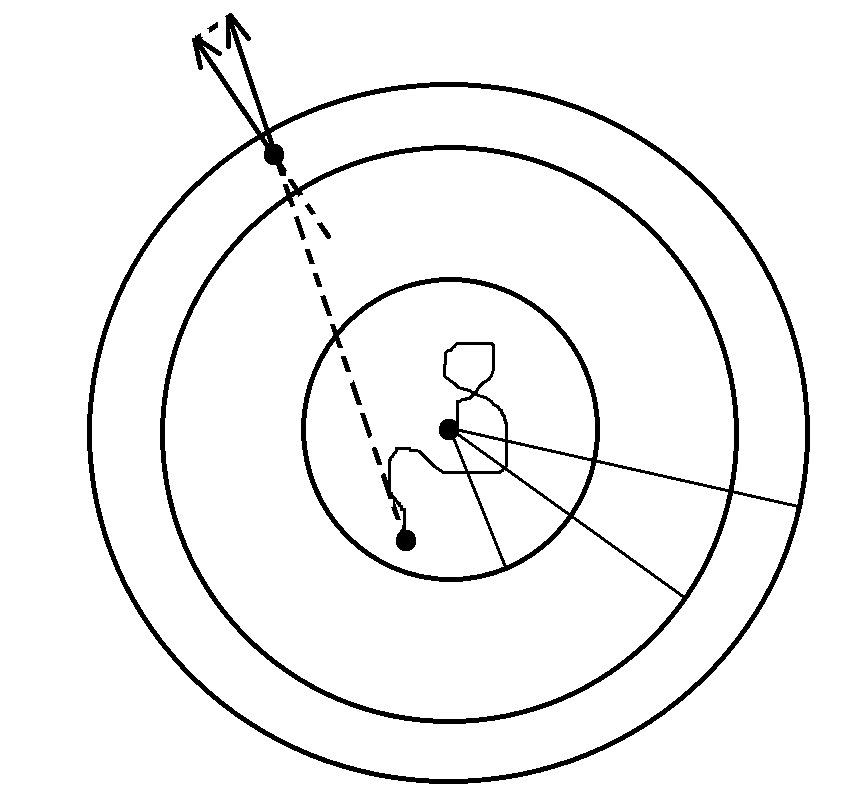}}
	\put(3.05,5.1){$y$}
	\put(3.2,6){drift}
	\put(1.6,5.9){radial}
	\put(1.4,5.6){component}
	\put(1.6,5.3){of drift}
	\put(4.35,3.25){$B_0$}
	\put(3.95,2.2){$B_t$}
	\put(5.2, 2.3){$\frac{1}{2}$}
	\put(6.1, 1.7){1}
	\put(6.4, 2.9){$\frac{5}{2}$}
	\end{picture}
	\caption{}
	\label{fig_angle}
\end{figure}

In order to make this argument precise, let $\mathbf{A}$ be the event that 
$\sigma>C_1/n$.  We will show that on the event $\mathbf{A}$, we have
$\rho_1=2\rho_0$.  Since Lemma \ref{lemma_1} shows that $P(\mathbf{A})>p$, 
choosing $p>2/3$ will then give us (\ref{eq:suffices-two-thirds}). 
From the reasoning which was given in the preceding paragraph and specifically in (\ref{drift-he}),
all that is left to verify is the statement about the radial part 
of the drift.  Let $a\in\mathbf{B}_{1/2}(0)$ and
$b\in\mathbf{B}_{5/2}(0)\setminus\mathbf{B}_{1}(0)$.  We wish to show that 
there is a constant $c>0$ not depending on $n$ such that for any such pair 
of points $a,b$ the projection of $b-a$ on $b$ has magnitude bounded below 
by $c$.  First we note that the points $0,a,b$ determine a plane passing 
through the origin, so we may assume that $n=2$, and our space is the $x-y$ 
plane.  We may also assume that $b$ lies on the vertical axis, with 
$y$-coordinate between $1$ and $5/2$.  Then the projection is bounded by
\begin{align} 
\label{eq-proj}
|(y-x)\cdot y|&=|y\cdot y-x\cdot y|  \\
&\geq \|y\|^2-\|x\|\cdot\|y\|  \nonumber\\
&\geq (\|y\|-\|x\|)\|y\|  \nonumber\\
&\geq \frac{1}{2}\cdot1. \nonumber
\end{align}
So, we may choose $c = 1/2$. This finishes the proof of Lemma \ref{lem:lower-bound-two-thirds}.
\end{proof}

\section{Proof of Theorem \ref{theorem-hit}} 

We again assume that $n\geq2$.  
First we introduce some definitions and notation which will help us to 
compare the problem of hitting the origin for $\phi^{\mathcal A_0}$, to 
the corresponding problem for a class of biased random walks. 
\paragraph{\textbf{Notation.}}
Recall our previous notation $x= (x_{1},...,x_{n})$. We define 
$x^{\bot}=(x_{2},...,x_{n})$. We further write
$B_{t}=(B^{(1)}_{t},...,B^{(n)}_{t})$ and define
$B_{t}^{\bot}= (B^{(2)}_{t},...,B^{(n)}_{t})$.

As stated in Theorem \ref{theorem-hit}, we assume
\begin{equation}
\label{eq:assumption-half-space}
\mathcal{A}_0=\{x\in\mathbb{R}^n: x_1\geq1\}.
\end{equation}
By scaling via Lemma \ref{lem:scaling}, if we prove Theorem 
\ref{theorem-hit} for this definition of $\mathcal{A}_0$, we have also 
proved it for 
$\mathcal{A}_0=\{x\in\mathbb{R}^n: x_1\geq a\}$ for any $a>0$.  

Next, we define the following objects by induction.  
\begin{enumerate}
\item[(i)] A sequence $\tau_0\leq\tau_1\leq\cdots$ of stopping times with 
limit \newline
$\tau_\infty=\lim_{i\to\infty}\tau_i$.  
\item[(ii)] A sequence of positive random variables $\rho_0,\rho_1,\ldots$.
\item[(iv)] A collection $\{\mathbb{A}_t\}_{t<\tau_\infty}$ of random subsets of $\mathbb{R}^n$. 
\end{enumerate}
Here, $\rho_i$ is the distance of $B_{\tau_i}$ to $\mathbb{A}_{\tau_i}$. Recall that $\tau$ was defined in (\ref{tau-def-phi}). Note that it is possible that $\tau<\tau_\infty$.  
One of the main ingredients in the induction is to choose $\mathcal{A}_t$ such that, $\mathbb{A}_t\subset\mathcal{A}_t$ for $t<\tau\wedge\tau_\infty$. 

To begin with, let $\tau_0=0$ and let
\begin{equation*}
\rho_0=\sup_{x\in \partial\mathcal A_{0}} {|B_{0}^{(1)}-x_1|}=1,
\end{equation*}
where $ \partial\mathcal A_{0} = \{x\in\mathbb{R}^n: x_1 = 1\}$ is the boundary of the set $\mathcal A_{0}$.
Let $\tau_1$ be the first time $t>0$ such that 
\begin{equation} \label{tau1-def} 
\sup_{x\in \partial \mathcal A_{t}} {|B_{t}^{(1)}-x_1|}=\frac{1}{2}\text{ or }2
\end{equation}
and note that $\tau_1<\tau$ unless both of these times are $\infty$.  We 
leave it to the reader to check that $\tau_1<\infty$ with probability one.  
For $0\leq t<\tau_1$, let
\begin{equation*}
\mathbb{A}_t=\mathcal{A}_t.  
\end{equation*}
We provisionally define
\begin{equation*}
\tilde{\mathbb{A}}_{\tau_1}=\mathcal{A}_{\tau_1}.  
\end{equation*}
Finally, we let
\begin{equation*}
\rho_1=\sup_{x\in \partial \tilde{\mathbb{A}}_{\tau_1}} {|B_{\tau_1}^{(1)}-x_1|}.
\end{equation*}
See Figure~\ref{fig_setup2} for the setup.

\begin{figure}[!h]
	\centering
	\setlength{\unitlength}{0.1\textwidth}
	\begin{picture}(10,4)
	\put(0,0){\includegraphics[width=\textwidth]{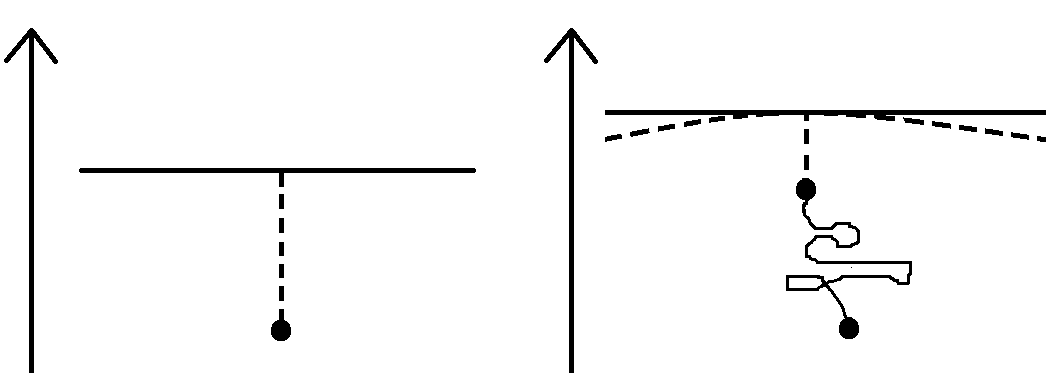}}
	\put(2.1, 2.4){$\mathcal{A}_0 = \mathbb{A}_0$}
	\put(2.8, 1.3){$\rho_0 = 1$}
	\put(5.7, 2){$\partial\tilde{\mathbb{A}}_{\tau_1}$}
	\put(5.7, 2.8){$\partial\mathbb{A}_{\tau_1}$}
	\put(7.8, 2.2){$\rho_1 = \frac{1}{2}$}
	\put(2.8, 0.3){$B_0 $}
	\put(7.8, 1.8){$B_{\tau_1}$}
	\put(0.15, 3.7){$x_1 $}
	\put(5.3, 3.7){$x_1$}
	\end{picture}
	\caption{The initial set $\mathbb{A}_0$ and the case when $\rho_1 = 1/2$.}
	\label{fig_setup2}
\end{figure}

Now assume that we have defined $\tau_m$, $\rho_{m}$, and 
$\{\mathbb{A}_t\}_{t < \tau_m}$ such that $\mathbb{A}_t\subset\mathcal{A}_t$ 
for $t < \tau_m$.  First we define
\begin{equation*}
\mathbb{A}_{\tau_m}=\{x\in\mathbb{R}^n: x_1\geq B^{(1)}_{\tau_m}+\rho_m\}
\end{equation*}
and note that $\mathbb{A}_{\tau_m}\subset\mathcal{A}_{\tau_m}$.  

Recall that the shift operator $\theta_t$ was defined before Lemma \ref{lem:strong-Markov}. Next we define $\tau_{m+1}$, so that $\tau_{m+1} -\tau_m$ is the first time $t>0$ such that 
\begin{equation*}
\sup_{x\in \partial (\theta_{\tau_m}\psi^{\mathbb{A}_{\tau_m}})_t(\mathbb{A}_{\tau_m})} {|B_{t}^{(1)}-x_1|}=\frac{1}{2}\rho_{m}\text{ or }2\rho_{m}.
\end{equation*}
Then, for $0\leq t<\tau_{m+1}-\tau_m$, let
\begin{equation*}
\mathbb{A}_{t+\tau_m}=(\theta_{\tau_m}\psi^{\mathbb{A}_{\tau_m}})_t(\mathbb{A}_{\tau_m}). 
\end{equation*}
Again, we provisionally define 
\begin{equation*}
\tilde{\mathbb{A}}_{\tau_{m+1}}=(\theta_{\tau_m}\psi^{\mathbb{A}_{\tau_m}})_{\tau_{m+1}}(\mathbb{A}_{\tau_m}).
\end{equation*}
and define
\begin{equation*}
\rho_{m+1}=\sup_{x\in \partial\tilde{\mathbb{A}}_{\tau_{m+1}}} {|B_{\tau_{m+1}}^{(1)}-x_1|}.
\end{equation*}

This finishes our inductive definition. Our next goal is to show the following proposition, which essentially proves  Theorem \ref{theorem-hit}. The rest of this section is devoted to the proof of this proposition. 
\begin{proposition} 
\label{prop-rw-hit}
We have 
\begin{itemize} 
\item[(i)] $\lim_{i\to\infty}\rho_i=0$ almost surely.  
\item[(ii)] $P(\tau_\infty<\infty)=1$.
\end{itemize} 
\end{proposition}

\begin{proof}
The proof of Proposition \ref{prop-rw-hit} is rather long, hence we divide it into a few steps.

But before presenting the details, we give a brief overview of the 
proof.  A point $x$ moves in the vertical direction (the $x_1$ 
direction) because Brownian motion $B_t$ spends time with 
$B^{(1)}$ close to $x_1$, and with $B^\perp$ close to $x^\perp$.  To 
take account of $|x^\perp-B^\perp|$, we cover the Brownian path 
$B_t^\perp$ with balls of various sizes, and consider the amount of time 
the Brownian motion spends in (i) balls of the smallest size and (ii) 
annuli consisting of set differences between two balls of the same 
center.  Secondly, as time progresses, the point $x$ may move not only 
in the vertical direction, but also in the lateral direction (the $x^\perp$ 
direction).  As $x$ moves in the lateral direction, it may get closer to 
certain parts of the Brownian path which did not contribute much drift 
previously.  Thus, we need to control the lateral distance over which $x$ 
might move.  Using the same balls as before, we will control the amount 
of lateral drift experienced by $x$, and thus the distance which 
$x^\perp$ can travel.  The upward drift experienced by $x_1$ will then 
be the sum of contributions from all the balls, and over all possible 
positions of $x^\perp$.  
 
\textbf{Step 1: a random walk comparison.} We set up a random walk 
comparison as we did at the beginning of Section \ref{sec:not-hit}. 
Recall that the vertical direction refers to the first coordinate of 
$\mathbb{R}^n$, and the horizontal direction refers to the remaining 
coordinates.  We will show that the vertical distance between the Brownian
motion and $ {\mathcal{A}}_{t}$ is bounded by the distance from the origin
of a one dimensional biased random walk initiated at $+1$. 

In order to do that we consider $\overline {\mathcal{A}}_{t}(x)$ which 
satisfies (\ref{main_equation}), with $\overline F(x) \equiv \|F\|_{\infty}$ 
for all $x\in \mathbb R$ instead of $F$.  
For every $i=0,1,...$, let $\overline\tau_{i}$ and $\overline \rho_{i}$ be 
the equivalents of $ \tau_{i}$ and $\rho_{i}$ with $\overline{\mathcal{A}}$ 
instead of $\mathcal{A}$. Moreover, let $\mathbb{\overline{A}}_{t}$ be the 
equivalent of $\mathbb{ A}_{t}$ with $\mathcal{\overline{A}}$ instead of $\mathcal{A}$. Then from 
the construction of $\{(\tau_{i},\rho_{i})\}_{i\geq 0}$ and $\mathbb{ A}_{t}$, 
it follows that for every $i=0,1,...$, 
\begin{equation*}  
 P\big(\overline \rho_{i+1} =  \overline \rho_{i}/2\big) \leq P\big(\rho_{i+1} =\rho_{i}/2\big),
\end{equation*}  
and 
\begin{equation*}
P(\overline \tau_\infty<\infty)  \leq P(\tau_\infty<\infty). 
\end{equation*}
From the preceding paragraph we conclude that it is enough to prove 
Proposition~\ref{prop-rw-hit} for $\overline{\mathcal{A}}$. 
In order to simplify notation, we will write
$(\mathcal{A}_{t}(x),  \tau_{i},\rho_{i}, \mathbb{ A}_{t})$ instead of 
$(\overline{\mathcal{A}}_{t}(x),  \overline \tau_{i}, \overline \rho_{i},  \mathbb{ \overline A}_{t})$,
when there is no ambiguity.    
 
Using scaling via Lemma \ref{lem:scaling} with $\|F\|_{\infty}$ 
in place of $F$, we find that there is a constant $p\in(0,1)$ such that,
conditioned on $\mathcal{F}_{\tau_{i-1}}$,
\begin{equation*}
\log_2\rho_i=
\begin{cases}
\log_2 \rho_{i-1}+1 & \text{with probability $p$} \\
\log_2 \rho_{i-1}-1 & \text{with probability $1-p$},
\end{cases}
\end{equation*} 
for $i=0,1,\ldots$, and the increments $\rho_{i+1}-\rho_i$ are independent.  
It follows that $\{\log_2 \rho_i\}_{i\in\mathbb{N}}$ is a 
nearest-neighbor random walk on the integers, and is biased if $p\ne1/2$. So, 
$\rho_i \to 0$ as $i \to \infty$ with probability 1 if and only if $p<1/2$.
Since the probability $p$ is the same for each time stage 
$[\tau_m,\tau_{m+1}]$, we prove that $p<1/2$ only on the first time stage 
$t\in[\tau_0,\tau_1]$, and this will finish the proof of 
Proposition \ref{prop-rw-hit}(i).

In addition to $\rho_1$, we define the maximum displacement of $\tilde{\mathbb{A}}_{\tau_1}$ in the $x_1$-direction $D_{\tau_1}$ as 
\begin{equation} \label{D-B1}
D_{\tau_1}=\sup_{x\in\partial \tilde{\mathbb{A}}_{\tau_1}}(x_1-1)
 =\rho_1+B^{(1)}_{\tau_1}-1.
\end{equation}

We are interested in finding a bound of the probability $p= P\big(\rho_{1} =2\big)$. First note that if we set $D_{\tau_1}\equiv 0$, from (\ref{D-B1}) and by the gambler's ruin, we would 
have $p=1/3$, since it is the probability that $1-B^{(1)}_t$, a 
one-dimensional Brownian motion starting at 1, hits 2 before hitting 1/2.
That is,
\begin{equation*}
p = \frac{1/2}{1/2+1} =1/3. 
\end{equation*}
Since $D_{\tau_1}> 0,$ $P$-a.s., in order to get a bound on $p$, we assume the following hypothesis, which will be verified later. For every $\delta>0$ sufficiently small,
\begin{equation}  \label{D-bound} 
P\left(D_{\tau_{1}} <\frac{1}{4} -2 \delta \right) >1-\delta.
\end{equation} 
From (\ref{D-B1}) it follows that under $\{D_{\tau_{1}} <\frac{1}{4} -2 \delta\}$,  $p$ is bounded from above by the probability that a Brownian motion, starting from $1+\frac{1}{4} -2 \delta$, exits the interval $[1/2, 2]$ from the right boundary.
It follows that 
\begin{equation*}
q \equiv 1-p \geq \frac{2-(1+\frac{1}{4} - 2\delta)}{2-(1+\frac{1}{4} - 2\delta)+|1/2-(1+\frac{1}{4} - 2\delta)|} >1/2+\delta. 
\end{equation*}
Therefore, in order to prove that $B_{t}$ hits $\mathbb{A}_{t}$, we just need to verify (\ref{D-bound}). The remainder of the proof is dedicated to deriving an inequality similar to (\ref{D-bound}). 

\textbf{Step 2: construction of a cover on the Brownian path.} 
Recall that $B_{0}=0$. In order to bound the vertical and lateral drifts we first need to fix a cover of the Brownian path $\{B^{\bot}_{t}:t\in [0,T]\}$
in $\mathbb{R}^{n-1}$, for some arbitrary $T>0$, and specify some of its properties.

In what follows 
$\lceil\cdot \rceil$ is the ceiling function which gives the smallest 
integer greater than or equal to the number inside.

\begin{proposition} \label{prop-cover-bm}
For any $\varepsilon>0$ there exists a constant $C_2$ not depending on $n$ such 
that the following holds.  If $K \in (0,1/8)$ and
$r=e^{k}$ with integer $1 \le k \le k_\infty$, where $k_\infty =\big\lceil (\log n)/2\big\rceil$, then with 
probability at least $1-\varepsilon$, we can 
cover $\{B^{\bot}_{t}:t\in [0,T]\} \subset \mathbb{R}^{n-1}$ with a number
$\big\lceil C_2\frac{nT}{Ke^{2k}} \big\rceil$ of $(n-1)$-dimensional balls of 
radius $r$.  
\end{proposition}  
We prove Proposition \ref{prop-cover-bm} in Section \ref{Section-pf-prop-cover}.  \\

Recall that $\tau_1$ was defined in \eqref{tau1-def}. We define the following event: 
\begin{equation} \label{Ev-1}  
\mathcal E_{1} = \{\tau_1 \leq T\}.
\end{equation}  
The following Lemma, which is proved in Section \ref{sec-pf-prop-lem}, helps us to bound the probability of $\mathcal E_{1}^{c}$.
\begin{lemma}  \label{lem-tau-bnd} 
There exist constants $C_{3}, C_{4} > 0$ such that   
\begin{equation*} 
P\big(\tau_1 >T\big)  \leq C_{3}\exp\left(-C_{4} T\right). 
\end{equation*}  
\end{lemma} 
Let $\varepsilon>0$ be arbitrarily small. By Lemma~\ref{lem-tau-bnd}, we can fix $T$ large enough so that 
\begin{equation}   \label{eps-1-ev} 
P(\mathcal E_{1}^{c}) < \varepsilon. 
\end{equation} 
We also fix $K \in (0,1/8)$ and let
\begin{equation}   \label{m-star} 
m_{k}^{*}(n)=\big\lceil  C_2\frac{nT}{Ke^{2k}} \big\rceil, 
\end{equation}  
where $C_{2}$ is defined as in Proposition \ref{prop-cover-bm}. 

We define $\mathcal E_{2}$ to be the event that there is a cover of 
$\{B^{\bot}_{t}:t\in [0,T]\}$ with $m_{k}^{*}(n)$ balls of of radius $e^{k}$ 
for any $k=1,...,\big\lceil (\log n)/2\big\rceil$. From Proposition 
\ref{prop-cover-bm} we have  
\begin{equation}    \label{eps-2-ev}
P(\mathcal E_{2}^{c}) < \varepsilon. 
\end{equation} 
\\
\textbf{Step 3: uniform bound on the Brownian occupation time.} In order to bound the vertical and lateral drift, we 
will bound the amount of time spent by the Brownian motion in each ball of the cover which was constructed in the previous step. 
Let $L^{(n-1)}_t(\cdot)$ be the occupation 
measure of the $n-1$ dimensional Brownian motion $B^{\bot}_{t}$. Define 
$L^{(n-1)}(\cdot) = \lim_{t \rightarrow \infty } L^{(n-1)}_{t}(\cdot)$, 
where we often omit the dependence in $n$ and write $L_t(\cdot)$ and $L(\cdot)$. 
We will need the following proposition. 
\begin{proposition} \label{prop-L-bound}
Assume that $n\geq 4$. Let $\mathbf B_{r}(0)$ be a ball of radius $r>0$ centered at the origin. Then, there exist constants $C_{5},C_{6}>0$ such that
\begin{equation*}
P\big(L^{(n-1)}\big(\mathbf B_{r}(0)\big)>sr ^{2}\big) \leq C_{5}\exp\left(-C_{6}n^2{s}\right)
\end{equation*}
for all $s > \frac{8}{n}$. 
\end{proposition} 
The proof of Proposition \ref{prop-L-bound} is given in Section \ref{sec-pf-prop-lem}. 
\\
From now on assume that $n\geq 4$. We will deal with the case where $n=2,3$ later. Assume further that $\mathcal E_{2}$ is satisfied. Denote by $\mathbf B_{i}^{k}$, $i=1,...,m^{*}_{k}(n)$, the balls of radius $e^{k}$ in the random cover of $\{B^{\bot}_{t}:t\in [0,T]\}$. 
We also write $C\cdot \mathbf{B}^{k}_i$ as the ball with the same center as $\mathbf{B}^{k}_i$ with radius $C$ times the radius of $\mathbf{B}^{k}_i$.

In what follows we fix $1 \le k \le k_\infty$, where $k_\infty =\big\lceil (\log n)/2\big\rceil$.  
Suppose that $p\in\mathbf{R}^{n-1}$, and define 
\begin{equation}
\label{eq:close-to-p}
\bar \tau(p)  = \inf\{t \geq 0 \, : \, |B_t-p|<2e^{k} \}. 
\end{equation}
Assume that $\bar \tau(p)\leq T$. 
Let $\tau_{0},\tau_1,\tau_2,\ldots$ be the times at which we choose new balls in the cover of $\{B^{\bot}_{t}:t\in [0,T]\}$.  That is, $\tau_{0} = 0$, and $\mathbf{B}_0^k$ is centered at the origin. Then, $\tau_{1}$ is the exit time from $\mathbf{B}_0^k$ and $\mathbf{B}_1^k$ is centered at $\mathbf{B}_{\tau_{1}}$. We continue in the same manner, so $\mathbf{B}_i^k$ has its center at $B_{\tau_i}$ with radius $e^k$. By our assumption $\bar \tau(p)\leq T$, and since $\mathbf{B}_i^k$ 
cover $B_{[0,T]}$, it follows that there must be a first index $i(p)$ such 
that 
\begin{equation} \label{b-tau1}
|B_{\tau_{i(p)}}-B_{\bar \tau(p)}|<e^{k}
\end{equation}
and hence
\begin{equation}  \label{b-tau2}
|B_{\tau_{i(p)}}-p|<3e^{k}.
\end{equation}
Note that $\tau_{i(p)}$ is a stopping time with respect to the Brownian 
filtration $(\mathcal{F}_t)_{t\geq0}$.  Also, we have $\tau_{i(p)}<\bar \tau (p)$ by 
the construction of our cover.  By this definition, 
$B_{t}\not\in2\mathbf{B}_{e^k}(p)$ for $t<\bar \tau(p)$, and so 
$B_{t}\not\in2\mathbf{B}_{e^k}(p)$ for $t<\tau_{i(p)}$ also.  
From (\ref{b-tau1}) and (\ref{b-tau2}), we have 
\begin{equation} \label{cont1} 
2\mathbf{B}_{e^{k}}(p)\subset5\mathbf{B}_{i(p)}^k.
\end{equation}
Next, we consider the occupation measure of the Brownian path after time 
$\tau_{i(p)}$.  Using the time shift operator $\{\theta_{t}\}_{t\geq 0}$, we define
\begin{equation*}
L^{(n-1)}_{\tau_{i(p)}}(2\mathbf{B}_{e^{k}}(p))
= \theta_{\tau_i(p)}\cdot L^{(n-1)}(2\mathbf{B}_{e^{k}}(p)).
\end{equation*}
Since $B_{t}\not\in2\mathbf{B}_{e^k}(p)$ for $t<\tau_i(p)$, it follows from (\ref{cont1}) that, 
\begin{equation} \label{ineq-L101}
\begin{aligned}
L^{(n-1)}(2\mathbf{B}_{e^{k}}(p))
&=L^{(n-1)}_{\tau_i(p)}(2\mathbf{B}_{e^{k}}(p))  \\
&\leq L^{(n-1)}_{\tau_i(p)}\big(5\mathbf{B}_{i(p)}^k\big).
\end{aligned}
\end{equation}
By the strong Markov property of Brownian motion with respect to the 
stopping time $\tau_i(p)$ and using Proposition \ref{prop-L-bound},  it follows that if $s>8/n$, then there are $C_{5},C_{6}>0$ such that 
\begin{equation} \label{ineq-L}
P\big(L^{(n-1)}_{\tau_i(p)}(5\mathbf{B}_{i(p)}^k)>25se^{2k}\big)
\leq C_{5}\exp\left(-C_{6}n^2s\right).
\end{equation}
Now let $\mathcal{A}_{k,i}$ be the event that 
\begin{equation}  \label{b1}
L^{(n-1)}_{\tau_i(p)}(5\mathbf{B}_i^k)\leq \rho\frac{e^{2k}}{n}
\end{equation}
for some positive constant $\rho$ independent of $i$ and $k$, which will be fixed later.
From (\ref{ineq-L}) we have for $\rho/(25n) > 8/n$ (that is $\rho>200$),
\begin{equation} \label{A-k-bound} 
P(A^c_{k,i})\leq C_{5}\exp\left(-C_{6}n\rho/25\right).
\end{equation}
Let us define the good event where all bounds such as (\ref{b1}) are satisfied as
\begin{equation*}
\mathcal E_{3}=\bigcap_{k=1}^{k_\infty}\bigcap_{i=1}^{m^{*}_{k}(n)} A_{k,i}.
\end{equation*} 
Note that from (\ref{b1}) and (\ref{ineq-L101}), it follows that if $\rho > 200$, then we have on $\mathcal{E}_3$  
\begin{equation}  \label{ball-sup}
\sup_{p\in\mathbf{R}^{n-1}}L^{(n-1)}(2\mathbf{B}_{e^{k}}(p))
 \leq\rho\frac{e^{2k}}{n}.
\end{equation} 
From (\ref{ball-sup}) we get that under $\mathcal E_{3}$ we have a uniform bound in $p$ on the occupation measure $L^{(n-1)}$, which is a stronger statement than the local bound in Proposition \ref{prop-L-bound}.

We now show that we can choose $\rho$ large enough, so the probability of $\mathcal E_{3}$ will be arbitrary close to 1. Recall that $k_{\infty}= \big\lceil (\log n)/2\big\rceil$. 
 From 
(\ref{A-k-bound}) and (\ref{m-star}) it follows that 
\begin{align*}
P(\mathcal E_{3}^c)& \leq \sum_{k=1}^{k_\infty}\sum_{i=1}^{m^{*}_{k}(n)}P(A_{k,i}^c) \\
&\leq  \sum_{k=1}^{k_\infty}m^{*}_{k}(n) C_{5} e^{-C_{6}n\rho/25} \\
&= \sum_{k=1}^{k_\infty}\big\lceil C_2 \frac{nT}{Ke^{2k}} \big\rceil
C_{5} e^{-C_{6}n\rho/25}.
\end{align*}
Using (\ref{m-star}), we continue with
\begin{align*}
P(\mathcal E_{3}^c)&\le C \frac{nT}{K} 
e^{-C_{6}n\rho/25}\sum_{k=1}^{k_\infty}\frac{1}{e^{2k}}\\
&\le C\frac{nT}{K} 
e^{-C_{6}n\rho/25}\left(\frac{1}{e^2}+\int_{1}^{k_\infty}\frac{1}{e^{2x}}dx\right)\\
&\le C\frac{nT}{K} 
e^{-C_{6}n\rho/25}\left(\frac{1}{e^2}+\int_{1}^{(\log n)/2+1}\frac{1}{e^{2x}}dx\right)\\
&= C\frac{nT}{K} 
 e^{-C_{6}n\rho/25}\frac{1}{2e^2}\left(3-\frac{1}{n}\right)\\
&= C'\frac{T}{K} 
 e^{-C_{6}n\rho/25}\left(3n-1\right).
\end{align*}
Note that the right hand side is positive since $n\ge1$. In order to have $P(\mathcal E_{3}^c)<\varepsilon$, we need  $\rho$ to satisfy 
\begin{equation} \label{cond-c} 
C'\frac{T}{K} 
 e^{-C_{6}n\rho/25}\left(3n-1\right)<\varepsilon. 
\end{equation} 
From the conditions on (\ref{A-k-bound}) we can choose $\rho>200$ such that (\ref{cond-c}) is satisfied. For such $\rho$, we get
\begin{equation} \label{eps-3-ev} 
P(\mathcal E_{3}^c) <\varepsilon. 
\end{equation}

In the next steps we derive an upper bound of the lateral drift (in the $x^{\perp}$--direction) and the vertical drift ($x_{1}$--direction), during the time interval $[0,\tau_1]$, on point $p$ which initially is in $\mathcal{A}_{0}$. Our bound will be independent of the initial position $p$. 

 We write 
\begin{equation*}
p=(p_1,\ldots,p_n)=(p_1,p^\perp), 
\end{equation*}
where $p^\perp=(p_2,\ldots,p_n)\in\mathbf{R}^{n-1}$.  

\textbf{Step 4: bound on the lateral drift.} We use the cover of the Brownian path on $[0,T]$ and the bounds on the Brownian occupation time which were established in the previous steps, in order to bound the lateral drift (i.e. the displacement of $p^\perp$).   
We will use the notation $\{p^\perp(t)\}_{t\geq 0}$ to denote the position of $p^\perp$ at time $t$ 
with $p^\perp(0) \in \mathcal A_0$ being its original position. 

We first define concentric balls centered at $p(0)^{\bot}$ in $\mathbb{R}^{n-1}$.  
Let $V^{k}_{p(0)}$ be a ball of radius $2e^k$ centered at $p(0)^{\bot}$ for 
$k=1,...,k_\infty$. Then, we define the set of the concentric 
annuli $A^k_{p(0)} = V^{k}_{p(0)} \setminus V^{k-1}_{p(0)} $, for $k=2,3,...$, with 
the first set equal to the first ball: $A^1_{p(0)}:=V^1_{p(0)}$. So, these annuli 
(except for the first) have inner radii $2e^k$ and outer radii $2e^{k+1}$. 

Let $D^{\perp}_k$ be the maximal lateral drift experienced by a point 
$q=(q_1,q^\perp)$ such that   
\begin{equation} \label{point-q} 
q_1\leq 0 \quad \textrm{and } q^\perp\in\mathbf{B}_{e^{k}}(p^\perp(0))
\end{equation}
assuming that the Brownian particle is located at $r=(r_1,r^\perp)$. 
Recall that by the hypothesis of Theorem \ref{theorem-hit}, $\|F\|_{\infty}\leq c^*n^{3/4}$. On the first interval $[0,{\tau_1}]$, the vertical distance (along the $x_{1}$-axis) from 
$\mathbb{A}_{t}$ to $B_t$ is on the interval $[1/2,2]$ and so it is bounded below by 1/2. The lateral distance (along $x^{\bot}$) from $B_t$ to a point $q$ in 
$\mathbb{A}_{t}$ is $\|q^{\bot}-B^{\bot}_t\|$. 
From (\ref{main_equation}) it follows that, 
\begin{equation}  \label{let-drift} 
\begin{aligned}  
D^{\perp}_k 
&\le \|F\|_\infty 
 \frac{1}{\sqrt{1/4+\|q^{\perp}-r^{\perp}\|^2}} \\ 
&\leq c^*n^{3/4} 
 \frac{1}{\sqrt{1/4+\|q^{\perp}-r^{\perp}\|^2}}.
 \end{aligned} 
 \end{equation} 
We distinguish between the following two cases: 

\emph{Case 1}:  $r^\perp \in A^1_{p(0)}$. 
Then from (\ref{let-drift}) it follows that there exists a constant $C>0$ independent of $n$ and $p$, such that 
\begin{equation} \label{D-p-1} 
D^{\perp}_{1} \leq Cc^*n^{3/4}. 
\end{equation}
\emph{Case 2}:  $r^\perp \in  A^{k+1}_{p(0)}$, for $1 \leq k \leq k_\infty$. 
Then from (\ref{let-drift}) it follows that there exists a constant $C>0$ independent of $n$ and $p$, such that 
\begin{equation} \label{D-p-k} 
\begin{aligned} 
D^{\perp}_k&\leq \|F\|_\infty\frac{1}{\sqrt{1/4+e^{2k}}} \\
&\leq Cc^*n^{3/4}e^{-k}.
\end{aligned}
\end{equation}
We define $k_{0}$ by the equation
\begin{equation*}
e^{k_0}=\left(e^{k_\infty}\right)^{1/2}=n^{1/4}.
\end{equation*}
First we assume that the process $\{p^\perp(t)\}_{t\geq 0}$ stops when it exits from 
$\mathbf{B}_{e^{k_0}}(p^\perp(0))$.  If, under this assumption, we find that 
$p^\perp(t)$ in fact does not exit from $\mathbf{B}_{e^{k_0}}(p^\perp(0))$
over time $[0,T]$, then the unmodified process must not exit either.  

We will use our estimates on the occupation measure for Brownian motion from step 3, in the 
various balls $2\mathbf{B}_{e^{2k}}(p^\perp(0))$, in order to bound the horizontal drift experienced by points in 
$\mathcal A_t$ where $t\in [0,\tau_1]$.   
We distinct between the following 3 cases: $k=1$, $1<k<k_0$, and $k_0 \leq k \leq k_\infty$.  In the 
first two cases $k<k_0$, we have that $p^\perp(t)$ traverses at most 
$e^{k_0-k}$ balls.  For $k>k_0$, we have that $p^\perp(t)$ traverses at most 
one ball.  Then, recalling that $e^{k_0}=n^{1/4}$, we have the following.  
In each term, we first write our upper bound for the drift from (\ref{D-p-1}) and (\ref{D-p-k}), then the number 
of balls traversed, and then the occupation measure $L^{n-1}$ of each ball from (\ref{ball-sup}).  At first, 
we separate these terms by dots. Let $D^\perp_{\tau_1}$ be the total horizontal drift experienced by $p(t)$ between $[0,\tau_1]$. It follows that 
\begin{align*}
D^{\perp}_{\tau_{1}} &\leq Cc^*n^{3/4}\cdot e^{k_0}\cdot
          L^{(n-1)}\big[A_{1}(p^\perp(0))\big] \\
&\qquad + \sum_{k=2}^{k_0-1}Cc^*n^{3/4}e^{-k}\cdot e^{k_0-k}\cdot
   L^{(n-1)}\big[A_{k}(p^\perp(0))\big]  \\
&\qquad + \sum_{k=k_0}^{k_\infty}Cc^*n^{3/4}e^{-k}\cdot1\cdot
   L^{(n-1)}\big[A_{k}(p^\perp(0))\big]  \\
&\leq Cc^*n^{3/4}e^{k_0}\rho\frac{1}{n}
 + Cc^*\sum_{k=0}^{k_0}n^{3/4}e^{k_0-2k}\rho\frac{e^{2k}}{n}  
 + Cc^*\sum_{k=k_0}^{k_\infty}n^{3/4}e^{-k}\rho\frac{e^{2k}}{n}  \\
&\leq Cc^*\rho n^{-1/4}e^{k_0}
 + Cc^*\sum_{k=0}^{k_0}n^{3/4}e^{k_0}\rho\frac{1}{n}  
 + Cc^*\sum_{k=k_0}^{k_\infty}n^{3/4}e^{-k}\rho\frac{e^{2k}}{n}  \\
&\leq Cc^*\rho + Cc^*\rho k_0
 + Cc^*\rho n^{-1/4}e^{k_\infty}  \\
&\leq n^{1/4},
\end{align*}
for $c^*$ small enough. Note that we used $e^{k_0}=n^{1/4}$, 
$e^{k_\infty}=n^{1/2}$, and $k_0=C\log n$ in the last inequality. 

Because $\mathbf{B}_{e^{k_0}}(p^\perp(0))$ has radius $n^{1/4}$, we find 
that for $c^*$ small enough, $p^\perp(t)$ does not exit from this ball, as 
required.  

\textbf{Step 5: bound on the downward drift.}
We use the same strategy as in Step 4 to bound 
the total downward drift.   
Let $D^{1}_k$ be the maximal downward drift experienced by a point 
$q=(q_1,q^\perp)$  as in (\ref{point-q}), 
assuming that the Brownian particle is located at $r=(r_1,r^\perp)$. As in step 4, on the first interval $[0,{\tau_1}]$, the vertical distance (along the $x_{1}$-axis) from 
$\mathbb{A}_{t}$ to $B_t$ is on the interval $[1/2,2]$ and so it is bounded below by 1/2. The lateral 
distance (along $x^{\bot}$) from $B_t$ to $q$ is $\|q^{\bot}-r^{\bot}\|$.  We want to compute an upper bound of the 
vertical drift on the point $q$, so we multiply the factor 
$ \frac{1}{\sqrt{1/4+\|q^{\bot}-r^\bot \|^2}}$, which bounds the projection of 
the drift vector from $B_t$ to  $q$  onto the $x_{1}$--direction, to the diagonal drift. We get that 
\begin{equation}  \label{hor-drift} 
\begin{aligned}  
D^{1}_k 
& \le \|F\|_\infty 
 \frac{1}{\sqrt{1/4+\|q^{\perp}-r^{\perp}\|^2}}\cdot \frac{1}{\sqrt{1/4+\|q^{\perp}-r^{\perp}\|^2}} \\ 
&\leq c^*n^{3/4} 
 \frac{1}{1/4+\|q^{\perp}-r^{\perp}\|^2}.
 \end{aligned} 
 \end{equation} 
We again distinguish between the following cases:

\emph{Case 1}:  $r^\perp \in A^1_{p(0)}$. 
Then from (\ref{hor-drift}) it follows that there exists a constant $C>0$ independent of $n$ and $p$, such that 
\begin{equation} \label{D-1-1} 
D^{1}_{1}\leq  \|F\|_\infty\leq Cc^*n^{3/4}. 
\end{equation}
\emph{Case 2}:  $r^\perp \in  A^k_{p(0)}$, for $2 \leq k \leq k_\infty$. 
Then from (\ref{hor-drift}) it follows that there exists a constant $C>0$ independent of $n$ and $p$, such that 
\begin{equation}
\begin{aligned} \label{D-1-k} 
D^{1}_k&\leq c^*n^{3/4}\frac{1}{{1/4+e^{2k}}} \\
&\leq Cc^*n^{3/4}e^{-2k}.
\end{aligned}
\end{equation} 
Let $D^{1}_{\tau_{1}}$ be the total vertical drift experienced by $p(t)$ between $[0,\tau_1]$. As in the previous step, in each term, we first write our upper bound for the drift from (\ref{D-1-1}) and (\ref{D-1-k}), then the number of balls traversed, and then the occupation measure of each ball from (\ref{ball-sup}). It follows that 
\begin{align*}
D^{1}_{\tau_{1}} &\leq Cc^*n^{3/4}\cdot e^{k_0}\cdot
          L^{(n-1)}\big[A_{1}(p^\perp(0))\big] \\
&\qquad + \sum_{k=2}^{k_0-1}Cc^*n^{3/4}e^{-2k}\cdot e^{k_0-k}\cdot
   L^{(n-1)}\big[A_{k}(p^\perp(0))\big]  \\
&\qquad + \sum_{k=k_0}^{k_\infty}Cc^*n^{3/4}e^{-2k}\cdot1\cdot
   L^{(n-1)}\big[A_{k}(p^\perp(0))\big]  \\
&\leq Cc^*n^{3/4}\cdot n^{1/4}\cdot\rho\frac{1}{n}
 + Cc^*\sum_{k=0}^{k_0}n^{3/4}e^{k_0-3k}
        \rho\frac{e^{2k}}{n}  \\
&\qquad + Cc^*\sum_{k=k_0}^{k_\infty}n^{3/4}e^{-2k}
        \rho\frac{e^{2k}}{n}  \\
&\leq Cc^*\rho
 + Cc^*\rho n^{-1/4}e^{k_0}\sum_{k=0}^{k_0}e^{-k}
 + Cc^*\rho n^{-1/4}k_\infty  \\
&\leq Cc^*\rho
 + Cc^*\rho n^{-1/4}e^{k_0}
 + Cc^*\rho n^{-1/4}k_\infty. 
\end{align*}
Since $e^{k_0}=n^{1/4}$ and $k_\infty=(\log n)/2$, we have
\begin{equation*}
D \leq Cc^*.
\end{equation*}

By choosing a small $c^*$ we get  
\begin{equation} \label{D1-8} 
D^{1}_{\tau_{1}} \leq 1/8,
\end{equation}
on the event $\cap_{i=1}^{3}\mathcal E_{i}$. Recall $D_{\tau_{1}}$ was defined in (\ref{D-B1}). From (\ref{D1-8}),   
 (\ref{eps-1-ev}), (\ref{eps-2-ev}) and (\ref{eps-3-ev}), we get (\ref{D-bound}) and we finish the proof for $n>3$.  
\\

\textbf{Step 6: bound $D_{\tau_{1}}$ for $n=2,3$.}  In this case we can bound the total downward drift 
at $x$ due to $B^\perp_{[0,T]}$ by $c^{*}C_{7}$, where $C_{7}$ is some positive constant, since the 
Brownian motion is at a positive distance from $x$ on $[0,\tau_{1}]$. Again, 
by choosing $c^*$ small enough, we get  
\begin{equation*}
 D_{\tau_{1}} <c^{*}C_{7}T \leq 1/8
\end{equation*}
on the event $\mathcal E_{1}$. From the above equation and (\ref{eps-1-ev}) we get (\ref{D-bound}) and we finish the proof for $n=2,3$.
\end{proof}  

\section{Proof of Proposition \ref{prop-cover-bm}} \label{Section-pf-prop-cover}

Recall that our goal is to get an upper bound, with a high probability, on the number of 
balls of radius $r$ needed to cover the Brownian path $B^{\bot}_{[0,T]}$.   

Let $\tau_{r,1}$ be the first time $t$ that the $(n-1)$-dimensional 
Brownian motion $B^{\bot}_{t}$ initiated at the origin exits the ball $B(0,r)$ centered at $0$ of radius $r$.  
At $\tau_{r,1}$, we start the Brownian motion from the origin over again and wait till the next time $\tau_{r,2}$ at which the restarted Brownian motion reaches distance $r$ 
from its starting point.  We also write $\tau_{r,0}=0$.  Let
\begin{equation*}
\sigma_{i}=\sigma_{r,i}=\tau_{r,i}-\tau_{r,i-1}.
\end{equation*}
Then $\sigma_i:i=1,2,\ldots$ are i.i.d. random variables. 

Let $N=N_{r,T}$ be the number of balls of radius $r$ needed to cover the 
Brownian path $B^\perp_{[0,T]}$.  By Brownian scaling, $N$ equals in law to the 
number of balls of radius 1 needed to cover the Brownian path $B^\perp_{[0,T/r^2]}$.  
So we can redefine $\sigma_i$ and $N$ to reflect this Brownian scaling, that 
is, $\sigma_i$ are i.i.d. exit times from a unit ball.  

We expect $\sigma_i$ to be about $C/n$, since the Bessel process $\|B^{\bot}_t\|$ has 
drift about $n$ for $1/2\leq\|B^{\bot}_t\|\leq1$ and larger when $\|B^{\bot}_t\|\leq1/2$.    
So, roughly speaking we expect,  
\begin{equation}
\label{eq:hope-show}
N\approx \frac{nT}{r^2}. 
\end{equation}
The following lemma is an important ingredient in the proof of Proposition \ref{prop-cover-bm}. 
\begin{lemma}  \label{lemma-exit-t}
For any $m\geq 1$, there exist constants $C_{8},C_{9}>0$ such that 
\begin{equation*}
P\left(\inf_{i=1,\ldots,m}\sigma_i<K/n\right)
\leq mC_{8}\exp\left(-C_{9}n/K\right)
\end{equation*} 
for all $0< K <1/8$. 
\end{lemma} 
\begin{proof} 
Note that for any $K>0$ and $m\geq 1$ we have, 
\begin{equation*}
P\left(\inf_{i=1,\ldots,m}\sigma_i<K/n\right)
\leq mP\left(\sigma_1<K/n\right). 
\end{equation*}
Hence we need to show that $mP(\sigma_1<K/n)$ is small enough.
 
For simplicity, we let $\sigma=\sigma_1$.  Recall that the Bessel process $Y_t=\|B^{\bot}_t\|$ satisfies the following SDE
\begin{equation*}
dY_{t}=\frac{n-2}{2Y_{t}}dt+dW_{t}, 
\end{equation*}
where $W_{t}$ is a standard one dimensional Brownian motion. 

Note that $\sigma$ is greater than $\tilde{\sigma}$, which is the time 
needed for $Y_t$ to exit the interval $[1/2,1]$ starting from $Y_0=3/4$. 
Therefore, it suffices to bound
\begin{equation*}
mP\left(\tilde{\sigma}<K/n\right).
\end{equation*}
Suppose now that
\begin{equation*}
K < 1/8
\end{equation*}
so that within time $t<K/n$, the drift term in $Y_t$ is at most $1/8$.  
Then, in order for $\tilde{\sigma}<K/n$, we must have 
\begin{equation*}
\sup_{0\leq t\leq K/n}|W_t|\geq 1/8.
\end{equation*}
However, a standard Brownian estimate gives
\begin{align*}
P\left(\sup_{0\leq t\leq K/n}|W_t|\geq 1/8\right)
\leq C_{8}\exp\left(-C_{9}n/K\right), 
\end{align*}
where $C_{8},C_{9}$ are positive constants. 

It follows that 
\begin{equation*}
P\left(\inf_{i=1,\ldots,m}\sigma_i<K/n\right)
\leq mC_{8}\exp\left(-C_{9}n/K\right).
\end{equation*}
\end{proof} 

\paragraph{\textbf{Proof of Proposition \ref{prop-cover-bm}}} \label{sec-pf-prop-L-bound}
Recall that $r=e^{k}$, $k=1,...,k_\infty$, so we are covering the Brownian path with balls of radius $e^{k}$.  \\
By a scaling argument this is equivalent to a cover the path of $B^{\bot}$, between $0$ to $T/r^2=Te^{-2k}$ with balls of radius $1$. 
Note that if $\sigma_i\geq K/n$ for all $i=1,2,..$, then we will need
\begin{equation*}
m^*=m^*(n)=
\big\lceil\frac{nT}{Ke^{2k}} \big\rceil,
\end{equation*}
such balls. Let $0< K <1/8$. From Lemma \ref{lemma-exit-t} it follows that the probability to cover the path with $m^{*}$ balls of radius $1$ is
\begin{equation*}
P_{k}:=P\left(\inf_{i=1,\ldots,m*}\sigma_i<K/n\right)
\leq C_{8} \big\lceil\frac{nT}{Ke^{2k}} \big\rceil
\exp\left(-C_{9}n/K\right).
\end{equation*}
The sum of the probabilities $P_k$ over $k$ is bounded by
\begin{align*} 
\sum_{k=1}^{\lceil \frac12\log n \rceil	}P_{k} &\leq C_{8}\exp\left(-C_{9}n/K\right) \sum_{k=1}^{\lceil	 \frac12\log n \rceil	} \big\lceil\frac{nT}{Ke^{2k}} \big\rceil
 \\
&\leq  \widetilde C_{8}\frac{n T}{K}\exp\left(-C_{9}n/K\right).
\end{align*} 
It follows that for any $\varepsilon \in (0,1)$, there exists $N(\varepsilon,K)$ such that if $n>N(\varepsilon,K)$ then for any choice of $r=e^{k}$, $k=1,....\lceil	(\log n)/2\rceil	$, we have $\sigma_i \geq K/n$ for all $i=0,...,m^{*}$ with probability $1-\varepsilon$. From the preceding paragraph we get the result for  $n>N(\varepsilon,K)$. 

Note that similar covers in dimensions $n \leq N(\varepsilon,K)$ will hold by projection of the covers when $n > N(\varepsilon,K)$, possibly with an additional multiplicative constant in $m_{k}^{*}(n)$. We therefore proved the result for all $n\geq 1$.\qed 

\section{Proof of Proposition \ref{prop-L-bound} and Lemma \ref{lem-tau-bnd}} \label{sec-pf-prop-lem} 
Before we prove Proposition \ref{prop-L-bound}, we recall 
Theorem 2 (in page 444) from Ciesielski and Taylor
\cite{ciesielski-taylor62}.  
\begin{theorem}[Paraphrased] \label{C-T-thm} 
\label{th:ct}
Suppose $n\geq1$.  
Let $\tau_n$ be the first time that an $n$ dimensional Brownian motion leaves a ball 
of radius 1  in $\mathbb{R}^n$.  Let $L_n$ be the amount of time that the same Brownian motion spends in 
this ball.  Then
\begin{equation*}
\tau_n\stackrel{\mathcal{D}}{=}L_{n+2}.
\end{equation*}
\end{theorem}

\paragraph{\textbf{Proof of Proposition \ref{prop-L-bound}}} 
Recall that $B^{\bot}$ is an $n-1$ dimensional Brownian motion. From Theorem \ref{C-T-thm} we have for dimensions $n\geq 4$, 
\begin{align}
\label{eq:ct}
P(L_{n-1}>s)=P(\tau_{n-3}>s).
\end{align}
We will derive an upper bound on the right side of (\ref{eq:ct}).  For notational convenience, we will first establish  an upper bound on $P(\tau_n>s)$ as a function of $n$ and replace $n$ by $n-3$ later.

First of all, note that this bound 
only depends on $Y(t)=\|B(t)\|$ for $Y\in(0,1]$. In this case, the drift of $Y$ 
satisfies 
\begin{equation} \label{drift-comp} 
\frac{n-1}{2Y(t)}\geq\frac{n-1}{2}.
\end{equation}
Let 
\begin{equation*}
X(t)=\frac{(n-1)t}{2}+W(t),
\end{equation*}
where $W(t)$ is a one dimensional standard Brownian motion.

From (\ref{drift-comp}) and a standard comparison result (see Theorem 1.1 in Chapter V.1 of Ikeda and Watanabe \cite{iw89}), we compare the drifts of $Y$ and $X$, we see that $Y(t)>X(t)$ for $0\leq t\leq\tau_n$ with probability 1. Note that the result of Ikeda and Watanabe does not cover locally unbounded drift, but this could be argued via a 
standard truncation argument.   

Therefore, from (\ref{eq:ct}) we get, 
\begin{equation*}
P(L_{n+2}>s)\leq P(\tau^X_n>s),
\end{equation*}
where $\tau^X_n$ denotes the first time $t$ that $X_t=1$.  

Simple calculations give us 
\begin{align*}
P\big(\tau^X_n>s\big)&\leq P\big(X(s)\leq1\big)   \\
&=P\big((n-1)s/2+W(s)\leq1\big)  \\
&=P\big(W(s)\leq1 -(n-1)s/2\big).
\end{align*}
Note that for $s>4/(n-1)$, 
\begin{equation*}
1-\frac{(n-1)s}{2} \leq -\frac{(n-1)s}{4}.
\end{equation*}
So, we conclude that 
\begin{align}
\label{L-tail} 
P(L_{n+2}>s) &\le P\big(\tau^X_n>s\big)  \\
&\leq P\Big(W(s)\leq -\frac{(n-1)s}{4}\Big)  \nonumber\\
&\leq C_{5}'\exp\left(-C_{6}'n^2{s}\right)\nonumber
\end{align}
for some constants $C_{5}', C_{6}'$. Replacing $n$ by $n-3$  only changes the constants
$C_{5}', C_{6}'$ in (\ref{L-tail}) but nothing else. 
 \qed 

\subsection{Proof of Lemma \ref{lem-tau-bnd}}
Recall that we must show that there exist constants $C_{3},C_{4}>0$ such 
that   
\begin{equation*} 
P\big(\tau_1 >T\big)  \leq C_{3}\exp\left(-C_{4} T\right). 
\end{equation*}  

We distinguish between the following two cases.  

\textbf{Case 1.} $F\equiv0$. Then $P\big(\tau_{1} >T\big) \leq P\big( \sup_{t \in [0,T]} |B^{(1)}_{t}| \leq 2\big)$, and the result follows by a standard Gaussian tail estimate. 

\textbf{Case 2.} $\|F\|_{\infty}>0$. 
Note that $\rho_{t}$, which is the maximal vertical distance between $\mathcal A_{t}$ and $W_{t}$, is given by
\begin{equation*}
D_{t} =\rho_t+B^{(1)}_{t}-1, 
\end{equation*}
where $D_{t}$ is the maximal vertical drift accumulated up to time $t$. From (\ref{main_equation}) with $\|F\|_{\infty}$ instead of $F$ and (\ref{tau1-def}) we get
\begin{equation*}
D_{t} \geq \frac{t}{2}\|F\|_{\infty} \quad \textrm{for all } 0\leq t\leq \tau_{1}. 
\end{equation*}
It follows that 
\begin{equation*}
\rho_{t}\geq 1+\frac{t}{2}\|F\|_{\infty} + W_{t} \quad \textrm{for all } 0\leq t\leq \tau_{1}, 
\end{equation*}
where $W_{t}$ is a one dimensional Brownian motion. 
Let $Z$ be a standard Gaussian random variable. It follows that there exist $C_{3},C_{4}>0$ independent from $n$ such that
\begin{align*}
P\big(\tau_{1} >T\big) & \leq  P\Big(\sup_{0\leq t\leq T}\rho_{t} \leq 2\Big) \\
&\leq P\Big(1+\frac{1}{2}\|F\|_{\infty}T + W_{T}  \leq 2\Big)   \\
&= P\Big(Z \leq \frac{1}{\sqrt T}  - \frac{1}{2}\|F\|_{\infty} \sqrt{T} \Big)   \\
&\leq C_{3}\exp\left(-C_{4}T\right).
\end{align*}
\qed

\section{Proof of Theorem \ref{theorem-hit2}}

We use our usual notation for a point $a\in\mathbb{R}^n$, that is 
$a=(a_1,\ldots,a_n)$ and $a^\perp=(a_2,\ldots,a_n)$.  Let 
$\mathbf{B}^{(n-1)}_r(a^\perp)$ be a ball in $\mathbb{R}^{n-1}$ of radius $r$ 
and center $a^\perp$, for $a\in\mathbb{R}^n$.  

Let $\mathbb{F}$ denote the set of Lipschitz functions 
$F:\mathbb{R}^n\to\mathbb{R}$ satisfying the hypotheses of Theorems 
\ref{theorem-hit} and \ref{theorem-hit2}, namely $0\leq F(x) \leq c^*n^{3/4}$ for 
all $x\in\mathbb{R}^n$.  The following Lemma will help us to prove Theorem \ref{theorem-hit2}. 
\begin{lemma}
\label{lem:uniform-in-F}
Let
\begin{equation*}
\mathcal{A}^{F,m}_0=\{1\}\times\mathbf{B}^{(n-1)}_m\big(0\big)
\end{equation*} 
and let $\mathcal{A}_t^{F,m}$ be the image of $\mathcal{A}_0^{F,m}$ under 
our flow $\psi_t^{\mathcal{A}^{F,m}_0}$.  Then
\begin{equation*}
\lim_{m\to\infty}\inf_{F\in\mathbb{F}}P(B_t\text{ hits }\mathcal{A}_t^{F,m})=1.
\end{equation*}
\end{lemma}
Note that Theorem \ref{theorem-hit} implies that the above limit is 1 for 
$F\in\mathbb{F}$ fixed.  Lemma \ref{lem:uniform-in-F} states that this 
convergence is uniform in $F\in\mathbb{F}$.  

\begin{proof}[Proof of Lemma \ref{lem:uniform-in-F}]
Let $(\psi^F_t,B_t)$ be a solution to (\ref{main_equation}) with drift $F \in \mathbb F$, which is initiated by $\mathcal A_0 = \{1\}\times \mathbb{R}^{n-1}$. We define $\sigma^F$ to be the first time that $B_t$ hits  $\mathcal{A}_t = \psi^F_t(\mathcal A_0)$.  
Define $\overline F \equiv c^*n^{3/4}$, and note that $F \leq \overline F$ for all $F\in \mathbb{F}$.  

Let $\varepsilon>0$ be arbitrary small. From the statement of Theorem 2 it follows that there exists $T_0$ large enough, such that for all $T>T_0$, we have
$$
P(\sigma^{ \overline F}>T) < \varepsilon.
$$
We observe that in the proof of Theorem \ref{theorem-hit} the only bound on $F$ which was used is $\|F\|_{\infty} \leq c^*n^{3/4}$, see the second paragraph of Step 1 in Proposition \ref{prop-rw-hit} and the bounds on $\mathcal D^{\perp}_{\tau_1}$ and $\mathcal D^{1}_{\tau_1}$ in Steps 4 and 5 of Proposition \ref{prop-rw-hit} (respectively). We therefore conclude  that for all $T>T_0$
\begin{equation}\label{eq:first-epsilon-est} 
\sup_{F\in \mathbb{F}} P(\sigma^{  F}>T) < \varepsilon. 
\end{equation} 
Now choose $m>0$ depending on $\varepsilon$ and the dimension $n$ such 
that 
\begin{equation}
\label{eq:second-epsilon-est}
P\Big(\sup_{0\leq t\leq T}\|B^{(n-1)}_t\|>m\Big)<\varepsilon.  
\end{equation}
Let $A^F_m$ be the event that $B_t$ hits $\mathcal{A}_t^{F,m} = \psi^F_t(\mathcal{A}^{F,m}_0)$.  Also, let 
$\partial\mathcal{A}_t^{F,m}$ denote the image of the boundary of 
the original ball $\mathbf{B}^{(n-1)}_m(0)$ under the same drift.  We also note that the drift 
experienced by  $\mathcal{A}_t^{F,m}$ is continuous as long as $B_t$ has not yet hit  $\mathcal{A}_t^{F,m}$.  

Note that as long as $\|B^{(n-1)}_t\|<m$, the drift induced on 
$\partial\mathcal{A}_t^{F,m}$ by $B_t$, when projected on 
$\mathbb{R}^{n-1}$, and further projected onto the radial direction in 
$\mathbb{R}^{n-1}$, points away from the origin.  It follows that the 
projection of $\partial\mathcal{A}_t^{F,m}$ onto $\mathbb{R}^{n-1}$ lies 
in the exterior of the ball $\mathbf{B}_m(0)$, and in fact encloses this 
ball.  By continuity, it follows that as long as $\|B^{(n-1)}_t\|<m$, we 
have $\mathcal{A}_t^{F,m}\supset\mathbf{B}^{(n-1)}_m(0)$.  

By the above and (\ref{eq:first-epsilon-est}), we have
\begin{equation*}
\sup_{F\in \mathbb{F}} P(A^F_m)>1-2\varepsilon  
\end{equation*}
and this completes the proof of Lemma~\ref{lem:uniform-in-F}.
\end{proof} 

\begin{proof}[Proof of Theorem \ref{theorem-hit2}]
By the hypothesis of Theorem \ref{theorem-hit2} we can assume that 
$\mathcal{A}_0$ contains a cylinder 
\begin{equation*}
\mathbf{C}_{\delta}(a)=(a_1,a_1+\delta)
 \times\mathbf{B}^{(n-1)}_{\delta}(a^\perp)
\end{equation*}
for some $a\in\mathbb{R}^{n}$ and some $\delta>0$.  
Furthermore, we can assume that $a_1>0$.  


Let $0<\varepsilon<a_1/2$, $\varepsilon_1>0$, and let $\sigma$ be the first 
time $t$ that $B_t^{(1)}=a_1-\varepsilon/2$.  It is easy to see that there 
is a positive probability that $\|B_\sigma^\perp\|<\varepsilon$ and 
$\sigma<\varepsilon_1$. Thus, if 
$\varepsilon_{1}$ is small enough and $\sigma<\varepsilon_1$, the integral of the 
drift from time 0 to time $\sigma$, acting on 
$\mathbf{C}_\delta(a)$ up to time $\sigma$ will be arbitrarily small.  It 
follows that there is a positive probability that 
$\|B_\sigma^\perp\|<\varepsilon$ and $\psi_\sigma^{\mathcal{A}_0}(\mathbf{C}_\delta(a))$ 
contains the disc 
$(a_1+\varepsilon/2)\times \mathbf{B}^{(n-1)}_{\delta/2}(a^\perp)$.  So, 
$B_\sigma^{(1)}=a_1-\varepsilon/2$ has distance $\varepsilon$ from the line 
$x_{1}=a_1+\varepsilon/2$, and if $\varepsilon<\delta/4$ then 
\begin{equation*}
\mathbf{B}^{(n-1)}_{\delta/2}(a^\perp)
\supset \mathbf{B}^{(n-1)}_{\delta/4}(B_\sigma^\perp).
\end{equation*}

Using Lemma \ref{lem:uniform-in-F}, choose a natural number $m$ such that 
\begin{equation}
\label{eq:assume-hit}
\inf_{F\in\mathbb{F}}P(B_t\text{ hits }\mathcal{A}_t^{F,m})>0.  
\end{equation}
Using the strong Markov property, let us start over at time $\sigma$.  Next, 
by translation and by our scaling (Lemma \ref{lem:scaling}), we find that 
our new Brownian motion starts at 0, and that
$(a_1+\varepsilon/2)\times\mathbf{B}^{(n-1)}_{\delta/4}(B_\sigma^\perp)$ has been 
transformed to a set containing $\{1\}\times\mathbf{B}_K^{(n-1)}(0)$ with 
$K>\delta/(4\varepsilon)$.  If $\varepsilon>0$ is 
small enough, we conclude that $K>m$ and so 
\begin{equation*}
\{1\}\times\mathbf{B}_K^{(n-1)}(0)
\supset \{1\}\times\mathbf{B}_m^{(n-1)}(0).
\end{equation*}
Using (\ref{eq:assume-hit}), we conclude that with positive probability, 
$B_t$ hits $\mathcal{A}_t^{F,K}$.  Since our original region $\mathcal{A}_t$, 
after scaling, contains $\mathcal{A}_t^{F,K}$, it follows that $B_t$ hits 
$\mathcal{A}_t$ with positive probability.  

This completes the proof of Theorem \ref{theorem-hit2}.
\end{proof}

\section*{Acknowledgements}
We are very grateful to anonymous referees for careful reading of the manuscript, and for a number of useful comments and suggestions that significantly improved this paper. We also thank Alexey Kuznetsov who proposed the problem that we solve in this paper.   


\bibliographystyle{plain}

\def\cprime{$'$} \def\cprime{$'$} \def\cprime{$'$}

\end{document}